\documentclass[10pt]{amsart}

\usepackage{amsmath,amssymb,verbatim,amsthm,hyperref,bbm, mathrsfs,amscd,pb-diagram}
\usepackage{aliascnt} 
\usepackage{appendix}
\DeclareMathAlphabet{\mathpzc}{OT1}{pzc}{m}{it}
\usepackage[all]{xy}
\usepackage{multirow,longtable}

\theoremstyle{plain}

\newtheorem{Thm}{Theorem}[section]

\newtheorem{Cor}{Corollary}[section]

\newtheorem{Prop}{Proposition}[section]

\newtheorem{Lem}{Lemma}[section]

\newtheorem{Remark}{Remark}[section]

\numberwithin{equation}{section}
\newcommand{\Card}[1]{\left\vert #1\right\vert} 
\newcommand{\Places}{\mathcal{P}} 
\DeclareMathOperator{\zfun}{\zeta} 
\DeclareMathOperator{\zint}{\mathcal{Z}} 
\DeclareMathOperator{\Lfun}{\mathcal{L}} 
\newcommand{\st}{\operatorname{\mathfrak{st}}} 
\newcommand{\Hecke}{\mathcal{H}} 
\newcommand{\coset}[1]{\left[ #1 \right]}  
\newcommand{\FNorm}[1]{\left\vert #1 \right\vert} 

\newcommand{\Ind}{\operatorname{Ind}}
\newcommand{\ind}{\operatorname{ind}}
\newcommand{\Hom}{\operatorname{Hom}}

\newcommand{\G}{\mathbb{G}}

\newcommand{\C}{\mathbb C}
\newcommand{\A}{\mathbb{A}}

\newcommand{\R}{\mathbb{R}}
\newcommand{\N}{\mathbb{N}}

\newcommand{\M}{\mathcal{M}}
\newcommand{\mO}{\mathcal{O}}

\newcommand {\integral}[1]{\int\limits_{{#1}(k)\backslash {#1}(\A)}}
\newcommand{\bk}[1]{\left(#1\right)} 
\newcommand{\bm}{\begin{multline*}}
\newcommand{\tu}{\end  {multline*}}

\DeclareMathOperator{\Id}{\mathbbm{1}} 
\DeclareMathOperator{\unif}{\varpi} 
\newcommand{\Sch}[1]{\operatorname{\mathcal{S}}\left( #1 \right)} 
\newcommand{\meas}{\mu} 
\newcommand{\less}{<}
\newcommand{\more}{>}
\newcommand{\conv}{*}
\renewcommand{\check}[1]{#1 ^{\vee}} 
\DeclareMathOperator{\Real}{\mathfrak{Re}} 
\newcommand{\piece}[1]{\left\{\begin{matrix} #1 \end{matrix}\right.} 
\newcommand{\set}[1]{\left\{ #1 \right\}} 
\newcommand{\mvert}{\mathrel{}\middle\vert\mathrel{}} 

\title{Rankin-Selberg integral with Non-unique model
for the standard $\Lfun$-function of $G_2$}
\author{Nadya  Gurevich and Avner Segal}
\address{School of Mathematics, Ben Gurion University of the Negev, POB 653, Be'er Sheva 84105, Israel}
\email{ngur@math.bgu.ac.il}
\email{avners@math.bgu.ac.il}

\makeatletter
\renewcommand\section{\@startsection{section}{1}{\z@}%
                                  {-3.5ex \@plus -1ex \@minus-.2ex}%
                                  {2.3ex \@plus.2ex}%
                                  {\center\normalfont\large\bfseries}}
\makeatother

\usepackage{calligra}
\usepackage[T1]{fontenc}

\begin{document}
\begin{abstract}
Let $\Lfun^{S}\bk{\pi,s,\st}$ be a partial $\Lfun$-function of degree 7 of a cuspidal automorphic representation $\pi$ of the exceptional group $G_2$. Here we construct a Rankin-Selberg integral for representations having certain Fourier coefficient.
\end{abstract}

\maketitle

\begin{center}
Mathematics Subject Classification: 22E55 (11F27 11F70 22E50)
\end{center}

\tableofcontents


\begin{flushright}
\calligra\bfseries
In memory of \\
Ilya Piatetski-Shapiro \\
and Stephen Rallis
\end{flushright}

\section{Introduction}
Until the late 1980's it was believed that a Rankin-Selberg integral must unfold 
to a unique model of the representation in order to be factorizable. 
By unique model we mean that the space of functionals on the representation space
with certain invariance properties is one-dimensional.
The most common is the Whittaker model, but other unique models such as Bessel model were also used.

In their pioneering work  \cite{MR965059} I. Piatetski-Shapiro \& S. Rallis interpreted an integral, earlier considered by A. Andrianov\cite{MR525651}, as an adelic integral that unfolds to a non-unique model. Although the functional is not factorizable, the integral is, since the local integral  produces the same $\Lfun$-factor for {\bf any} functional with the same invariance properties applied to a spherical vector. 

There are many examples of adelic integrals that unfold with non-unique models. Only few of them were shown to represent $\Lfun$-functions. Some more examples are detailed in \cite{MR2192819} and \cite{MR1361787}. All the examples rely on the knowledge of the generating function for the considered $\Lfun$-function.

In this paper we consider a new  Rankin-Selberg integral on the exceptional group $G_2$  and prove that it represents the standard $\Lfun$-function $\Lfun^{S}\bk{\pi,s,\st}$ of degree $7$ for cuspidal representations having certain Fourier coefficient along the Heisenberg unipotent subgroup. The candidate global integral was suggested by Dihua Jiang in the course of the work on \cite{MR1918673} and he also performed the unfolding. However, since the generating function for the $\Lfun$-function has not been known, the unramified computation was not completed. It is only now that we found a way to overcome this difficulty. To the best of our knowledge this is the first time the unramified computation is performed without explicit knowledge of the generating function.

The integral introduced here binds the analytic behaviour of $\Lfun^{S}\bk{\pi,s,\st}$ with that of a degenerate Eisenstein series of $Spin_8$ which was studied in \cite{MR1918673}. In the last section we use information on the poles of this Eisenstein series to show that for a cuspidal representation $\pi$ having certain Fourier coefficient, the non-vanishing of the theta lift of $\pi$ to the finite group scheme $S_3$ is equivalent to the $\Lfun$-function having a double pole at $s=2$. 

The Rankin-Selberg integral for the standard $\Lfun$-function of generic representions $G_2$ was constructed by D. Ginzburg in \cite{MR1203229}. Recently D. Ginzburg and J. Hundley \cite{UnPublishedGinzburg} have established the meromorphic continuation of $\Lfun^{S}\bk{\pi,s,\st}$ for any cuspidal representation $\pi$ using a doubling construction. Their integral representation shows that the set of  poles of $\Lfun^{S}\bk{\pi,s,\st}$ is contained in the set of poles of a degenerate Eisenstein series on the exceptional group of type $E_8$.


\section{Preliminaries}
Let $k$ be a number field and $\Places$ be its set of places. For any $\nu\in\Places$ denote by $k_\nu$ the local field associated to $\nu$. If $\nu\less\infty$ denote by $\mO_\nu$ the ring of integers of $k_\nu$ and by $q_\nu$ the cardinality of the residue field of $k_\nu$.
\subsection{The group $G_2$}
Let $G$ be the split simple algebraic group of the exceptional type $G_2$ defined over $k$ with torus $T$ and Borel subgroup $B$. Fix a root system of $G$ and denote by $\alpha$ and $\beta$ the short and the long simple roots respectively.
The Dynkin digram of $G$ has the form
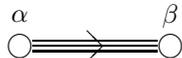
\begin{figure}[h!]
\label{fig:G2}
\begin{center}
\[
\xygraph{
!{<0cm,0cm>;<0cm,1cm>:<1cm,0cm>::}
!{(0.4,-1)}*{\alpha}="label1"
!{(0,-1)}*{\bigcirc}="1"
!{(0,0.1)}="c"
!{(0.2,  -0.1)}="c1"
!{(-0.2,-0.1)}="c2"
!{(0.4,1)}*{\beta}="label2"
!{(0,1)}*{\bigcirc}="2"
"1"-@3"2" "c1"-"c" "c"-"c2" 
}
\]
\end{center}
\end{figure}

and the set of positive roots is 
\[
\Phi^+=\{\alpha,\beta, \alpha+\beta,2\alpha+\beta,3\alpha+\beta,3\alpha+2\beta\} \ .
\]
The fundamental weights are denoted by
\[
\omega_1=2\alpha+\beta,\quad \omega_2=3\alpha+2\beta \ .
\]
For any root $\gamma$ fix a one-parametric subgroup $x_\gamma: \G_a\rightarrow G$. For any simple root $\gamma$ denote by $w_\gamma$ the simple reflection with respect to it, that is an element of the Weyl group of $G$. Define  also the coroot subgroups $h_{\gamma}: \G_m\rightarrow G$ such that for any root $\epsilon$
\[
\epsilon\bk{h_\gamma\bk{t}}=t^{<\epsilon,\check{\gamma}>} \ .
\]


\subsection{The partial $\Lfun$-function}
The dual group of $G$ is isomorphic to $G_2(\C)$. Denote the seven-dimensional complex representation of $G_2(\C)$ by $\st$. For an irreducible cuspidal representation $\pi=\otimes_v \pi_v$, unramified outside of  a finite set of places $S$, the standard partial $\Lfun$-function of $\pi$ is defined  by:
\[
\Lfun^S\bk{s,\pi,\st}=\prod_{v\notin S} \frac{1}{\det\bk{I- \st(t_{\pi_v}) q_v^{-s}}} \ .
\]
Here $t_{\pi_v}$ is the Satake parameter of $\pi_v$.

\subsection{Fourier coefficients}
The group $G$ contains a Heisenberg parabolic subgroup $P=M\cdot U$. The Levi part $M$ is isomorphic to $GL_2$ generated by the simple root $\alpha$, while $U$ is a five dimensional Heisenberg group. We denote the elements of $U$ by
\[
u(r_1,r_2,r_3,r_4,r_5):= 
x_\beta\bk{r_1} x_{\alpha+\beta}\bk{r_2} x_{2\alpha+\beta}\bk{r_3} x_{3\alpha+\beta}\bk{r_4} x_{3\alpha+2\beta}\bk{r_5} \ .
\]

The group $M$ acts naturally on $U$ and hence on $\Hom\bk{U, \G_a}$. It was shown in \cite{MR1637485} that for any field $F$ of characteristic zero the $M\bk{F}$-orbits of $\Hom\bk{U\bk{F},F}$ are naturally parametrized by isomorphism classes of cubic $F$-algebras.

Fixing an additive complex unitary character $\psi=\otimes_\nu \psi_\nu$ of $k\backslash \A$ this  give rise to the correspondence between $M\bk{k}$-orbits of complex characters of $U\bk{k}\backslash U\bk{\A}$ and cubic algebras over $k$. Let us denote by $\Psi_s$ the character corresponding to the {\bf split} cubic algebra $k\times k\times k$ and call it the {\bf split character}. More explicitly, 
\[
\Psi_s\bk{u\bk{r_1,r_2,r_3,r_4,r_5}}=\psi\bk{r_2+r_3} \ .
\]
Its stabilizer $S_{\Psi_s}$ in $M\bk{k}$ is isomorphic to $S_3$ and is generated by $w_{\alpha}$ and $ h_{\alpha}\bk{-1} x_{\alpha}\bk{-1}x_{-\alpha}\bk{1}$.

Denote by $\mathcal{A}\bk{G}$ the space of automorphic forms on $G$. For any  form $\varphi$ in $\mathcal{A}\bk{G} $ and complex character $\Psi$ of $U\bk{k}\backslash U\bk{\A}$ define the Fourier coefficient of $\varphi $ with respect to $\bk{U,\Psi}$ by
\[
L_{\Psi}\bk{\varphi}\bk{g}=\integral{U} \varphi\bk{ug}\overline{\Psi\bk{u}}\,du \ .
\]
For any $g\in{G}$ $L_\Psi\bk{\cdot}\bk{g}$ defines a functional in $\Hom_{U\bk{\A}}\bk{\mathcal{A}\bk{G},\C_\Psi}$.

For an automorphic representation $\pi$ of $G\bk{\A}$ we say that $\pi$ 
supports a $\bk{U,\Psi}$ coefficient if there exists a function $\varphi$ from the underlying space of $\pi$ such that $L_{\Psi}\bk{\varphi}\not\equiv 0$.

It was shown in \cite[Theorem 3.1]{MR2181091} that for any cuspidal representation $\pi$ there exists an \'etale cubic algebra such that $\pi$ supports a Fourier coefficient with respect to this algebra. Conversely, in \cite{MR1918673} it was shown that for any \'etale cubic algebra there exists a cuspidal representation supporting the Fourier coefficient corresponding to it. In this paper we consider only representations that support the split Fourier coefficient.

For a finite $\nu\in\Places$ denote by $K_{\nu}$ the maximal compact
subgroup $G\bk{\mO_v}$ of $G\bk{k_v}$ and by $\Hecke_{\nu}$ the spherical Hecke algebra. Given a complex character $\Psi$ of $U\bk{k_v}$ define 
\begin{align*}
\mathcal{M}_\Psi &=\set{f:G\bk{k_v}\rightarrow \C \mvert  f\bk{ugk}=\Psi\bk{u}f\bk{g} \quad \forall u\in U\bk{k_{\nu}}, k\in K_\nu} \\
\mathcal{M}^0_\Psi& =\set{ f:G\bk{k_v}\rightarrow \C \mvert f\bk{sugk}=\Psi\bk{u}f\bk{g}\quad \forall u\in U\bk{k_\nu},s\in S_\Psi, k\in K_\nu}  \ .
\end{align*}

For $f\in \Hecke_{\nu}$ define its Fourier transform $f^\Psi$ with respect to the character $\Psi$ by
\[
f^\Psi\bk{g}=\int\limits_{U\bk{k_v}} f\bk{ug}\overline{\Psi\bk{u}}\,du \ .
\]
Obviously $f^\Psi$ belongs to $\mathcal{M}^0_{\Psi}$.

\subsection{ The group $Spin_8$}
Let $H$ be a simply connected algebraic group of type $D_4$. We label its simple roots according to the following diagram.
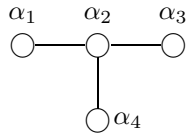
\begin{figure}[h!]
\label{fig:D4}
\begin{center}
\[
\xygraph{
!{<0cm,0cm>;<0cm,1cm>:<1cm,0cm>::}
!{(0,-1)}*{\bigcirc}="1"
!{(0.4,-1)}*{\alpha_1}="label1"
!{(0,0)}*{\bigcirc}="2"
!{(0.4,0)}*{\alpha_2}="label2"
!{(0,1)}*{\bigcirc}="3"
!{(0.4,1)}*{\alpha_3}="label3"
!{(-1,0)}*{\bigcirc}="4"
!{(-1,0.4)}*{\alpha_4}="label4"
"1"-"2" "2"-"3" "2"-"4"
}
\]
\end{center}
\end{figure}

The group of outer automorphisms of $H$ is isomorphic to $S_3$. Fixing one-parametric subgroups $x_\gamma: \G_a\rightarrow H$ defines a splitting of the sequence 
\[
1\rightarrow H^{ad}\rightarrow Aut\bk{H}\rightarrow Out\bk{H}\rightarrow 1 \ .
\]
In particular the semidirect product $H\rtimes S_3$ can be formed. It is well known that the centralizer of $S_3$ in $H\rtimes S_3$ is the group $G$. We identify $G$ with a subgroup of $H$ in this way. The group $H$ contains a maximal Heisenberg parabolic $P_H=M_H U_H$ such that $P=P_H\cap G$ given by
\[
M_H\simeq \set{\bk{g_1, g_2, g_3} \in GL_2\times GL_2\times GL_2 \mvert \det\bk{g_1}=\det\bk{g_2}=\det\bk{g_3}} \ .
\]
The modulus character of $P_H$ is given by $\delta_{P_H}\bk{g_1,g_2,g_3}=\FNorm{\det\bk{g_1}}^5$.

\subsection{The Eisenstein series}\label{ss_eisenstein}
Consider the induced representation $I_H\bk{s}:= Ind^{H\bk{\A}}_{P_H\bk{\A}} \delta_{P_H}^s$. All induced representations in this paper are not normalized. For any $K$-finite standard section $f_s$ define an Eisenstein series
\[
E\bk{g,f_s}=\sum_{\gamma \in P_H\bk{F}\backslash H\bk{F}} f_s\bk{\gamma g} \ .
\]
It has a meromorphic continuation to the whole complex plane. The behaviour at $s=4/5$ was studied
in  \cite{MR1918673}.

\begin{Prop}[\cite{MR1918673} Proposition 9.1]
\label{thm:eisenstein}
For any standard section $f$, the Eisenstein series $E\bk{g, f_s}$ 
has at most a double pole at $s=\frac{4}{5}$. The double pole is attained by the spherical section $f^0$. Also, the space
\[
Span_{\C}\set{\bk{s-\frac{4}{5}}^2 E\bk{g,f_s}\bigg{\vert}_{s=\frac{4}{5}}} \ ,
\]
is isomorphic to the minimal representation $\Pi$ of $H$.
\end{Prop}

It is customary to define the normalized Eisenstein series
\[
E^\ast\bk{g,f_s}= j\bk{s} E\bk{g,f_s},\text{ where } j\bk{s}=\zeta\bk{5s} \zeta\bk{5s-1}^2\zeta\bk{10s-4} \ .
\]

\section{The Zeta Integral}
\label{s_zetaintegral}
Let $\pi=\otimes \pi_v$ be an irreducible cuspidal representation of $G\bk{\A}$. For $\varphi\in\pi$ and $f_s \in I_H\bk{s}$ we consider the following integral,
\[
\zint\bk{s,\varphi,f}=\integral{G} \varphi(g) E^\ast (g,f_s)\, dg \ .
\]
Since $\varphi$ is cuspidal, and hence rapidly decreasing, the integral
defines a meromorphic function on the complex plane. Our main result is:

\begin{Thm}
\label{thm:Main}
Let $\pi=\otimes \pi_\nu$ be an irreducible cuspidal representation supporting the split Fourier coefficient. Let $\varphi=\otimes\varphi_\nu \in \pi$, $f_s=\otimes f_{s,\nu}\in I_H\bk{s}$ be factorizable data. Let $S\subset\Places$ be a finite set such that if $\nu\notin{S}$ then
\begin{itemize}
\item $\nu\not\vert{2,3,\infty}$
\item $\varphi_\nu$ is spherical
\item $f_{s,\nu}$ is spherical.
\end{itemize}
Then
\[
\mathcal{Z}\bk{s,\varphi,f}=\Lfun^S\bk{s,\pi, \st} d_S\bk{s, \varphi_S, f_S} \ .
\] 
Moreover for any $s_0$ there exist vectors $\varphi_S,f_S$ such that  $d_S\bk{s,\varphi_S,f_S}$ is analytic in a neighbourhood of $s_0$ and $d_S\bk{s_0,\varphi_S,f_S}\neq 0$.

In particular the partial $\Lfun$-function $\Lfun^S(s,\pi,\st)$ admits a meromorphic continuation.
\end{Thm}

\begin{Remark}
If $\pi$ does not support the split Fourier coefficient, the zeta integral vanishes identically. However if $\pi$ supports a Fourier coeficient corresponding to an \'etale cubic algebra $E$ there is a similar integral, using an Eisenstein series on the inner form of $Spin_8$ corresponding to $E$, that is expected to represent the same $\Lfun$-function. We plan to study these integrals in the near future.
\end{Remark}

The proof of the theorem will occupy the rest of the paper. We will explain main ideas, deferring the technical part to later sections and appendices.

\begin{Thm}[Unfolding]
\label{thm: unfolding}
For $\Real\bk{s}>>0$ we have
\begin{equation}
\label{eq:unfolded integral}
\zint\bk{s,\varphi,f}=
\int\limits_{U\bk{\A}\backslash G\bk{\A}} L_{\Psi_s}\bk{\varphi}\bk{g} F^*\bk{g,s}\, dg \ ,
\end{equation}
where 
\[
F^*\bk{g,s}=j\bk{s}\int\limits_{\A} f_s\bk{w_2w_3 x_{-\alpha_1}\bk{1}x_{\alpha+\beta}\bk{r}}\psi\bk{r} dr \ .
\]
\end{Thm}

This computation  was performed by Dihua Jiang,
 but since his proof was never published we include it in \autoref{s_unfolding}.

The function $F^*\bk{g,s}$ is factorizable whenever the involved section $f_s$ is. In particular 
\[
F^*\bk{g,s}=\Pi_\nu F^*_\nu\bk{g_\nu,s}, \quad \text{where} \quad F_\nu^*\bk{g,s}=j_\nu\bk{s} \int\limits_{k_\nu}  f_{s,\nu}\bk{\mu x_{\alpha+\beta}\bk{r} g} \psi_\nu\bk{r} dr \ ,
\]
and for almost all places $f_{s,\nu}=f_{s,\nu}^0\in Ind^{H\bk{k_\nu}}_{P_H\bk{k_\nu}}\delta_{P_H}^s$ is a  spherical vector with $f_{s,\nu}^0\bk{1}=1$.

Note that as the space $\Hom_{U(\A)}\bk{\pi, \C_{\Psi_s}}$ is usually infinite dimensional the functional $L_{\Psi_s}$ is not necessarily factorizable. Nevertheless it will be shown that the
integral $\mathcal{Z}\bk{\varphi,f,s}$ is factorizable. The factorizabily of the integral follows from the next surprising local statement, that replaces the unramified computation.

\begin{Thm}[Unramified Computation]
\label{thm:unramified}
Let $\pi_\nu$ be an irreducible unramified representation of $G\bk{k_\nu}$ and let $v^0$ be a fixed spherical vector in $\pi_\nu$. Assume that $\Hom_{U\bk{k_\nu}}\bk{\pi_\nu, \C_{\Psi_{s,\nu}}}\neq 0$. There exists $s_0\in\R$ such that for any $\Real{s}\more{s_0}$ and {\bf any} $l\in \Hom_{U\bk{k_\nu}}\bk{\pi_\nu, \C_{\Psi_{s,\nu}}}$ it holds
\begin{equation}
\label{unramified}
\int\limits_{U\bk{k_\nu}\backslash G\bk{k_\nu}} l\bk{\pi_v\bk{g} v^0} F^*_v\bk{g,s} dg =\Lfun\bk{5s-2,\pi_\nu,\st} l\bk{v^0} \ ,
\end{equation}
where $F_\nu^*\bk{g,s}$ is the function corresponding to the normalized spherical section $f^0_\nu$.
\end{Thm}

The identity in the main theorem follows from \autoref{unramified} using standard argument as in \cite{MR965059}. For the sake of completeness of presentation the argument is included in \autoref{PSRderive}. This argument also defines $d_S\bk{s,\varphi_S,f_S}$ explicitly. 

The proof of \autoref{thm:unramified} is the most non-trivial part of the paper and can be found in \autoref{sec:unramified}. In fact the proof is quite amusing. Following the ideas of \cite{MR965059} it boils down to proving the identity between $F_s$ and a Fourier transform of the generating function $\Delta$ of $L\bk{s,\pi,\st}$. We could not find the explicit formula for $\Delta$, which must be very complicated. Instead we have proven that the two functions become equal after being convolved with a third function. Both sides are evaluated explicitly (\autoref{App:Computing F(,s)} and \autoref{App:Computing D-psi and convolution}). Finally we show in \autoref{invertible} that the latter convolution is in fact an invertible operation.
\vspace{10pt}

\begin{Thm}[Ramified Computation]
\label{thm:ramified}
For any $s_0\in\C$ there exist datum $\varphi_S$ and $f_S$ such that $d_S\bk{s_0,\varphi_S,f_S}$ is holomorphic and non vanishing in a neighbourhood of $s_0$.
\end{Thm}
This theorem is proven in \autoref{ramified}.

\section{Unfolding}
\label{s_unfolding}
The proof of \autoref{thm: unfolding} is fairly standard. First we introduce some more notations that will be used in this section and also in \autoref{ramified}

Denote by $Q=LV$ the maximal parabolic subgroup of $G$ other than $P$. The Levi part $L\simeq GL_2$ is generated by the root $\beta$. The unipotent radical of the Borel subgroup of $L$ will be denoted by $N_\beta$. The unipotent radical $V$ of $Q$ is a three-step nilpotent group. Denote its commutator $\coset{V,V}$ by $R$.  It is generated by the subgroups $x_{2\alpha+\beta}, x_{3\alpha+\beta}$ and $x_{3\alpha+2\beta}$.

The following fact will be used (\cite[Theorem 5]{MR1020830}):
\begin{equation}
\label{rallis_equation}
\integral{R} \varphi\bk{rg} dr=\sum_{\nu \in N_\beta\bk{k}\backslash L\bk{k}} W_\psi\bk{\varphi}\bk{\nu g},
\end{equation}
where $W_\psi\bk{\varphi}$ is the standard Whittaker coefficient.

There are five $G\bk{k}$-orbits of $P_H\bk{k}\backslash H\bk{k}$. The representatives
of the orbits and their stabilizers are given in the next Lemma \cite[Lemma 2.1]{MR1617425}:
\begin{Lem}
The following is a list of representatives of $G\bk{k}$-orbits in	 $P_H\bk{k}\backslash H\bk{k}$ and their stabilizers:
\begin{enumerate}
\item $\mu=1$, and the stabilizer $G^\mu=P$.
\item $\mu=w_2w_1, w_2w_3, w_2w_4,$ and the stablizer $G^\mu=LR$.
\item $\mu= w_2w_3x_{-\alpha_1}(1)$ is a representative of the open orbit. The stabilizer of $P_H\bk{k}\mu G\bk{k}$ is $G^\mu=T^\mu\cdot U^\mu$ where
\[
T^\mu=\set{h_{3\alpha+2\beta}\bk{t}\mvert t\in k^\times},\quad
U^\mu=\set{u\bk{r_1,r_2,r_2,r_4,r_5}\mvert r_i \in k}
\]
\end{enumerate}
\end{Lem}

\begin{proof}[Proof of \autoref{thm: unfolding}]
For $\Real\bk{s}>>0$ it holds
\[
\integral{G} \varphi\bk{g} E\bk{g,f_s} dg= \integral{G} \varphi\bk{g} \sum_{\gamma \in P_H\bk{k}\backslash H\bk{k}} f_s\bk{\gamma g}dg= \sum_{\mu \in P_H\bk{k}\backslash H\bk{k}/G\bk{k}} I_\mu\bk{\varphi,f_s}\ ,
\]
where
\[
I_\mu\bk{\varphi,f_s}=\int\limits_{G^\mu\bk{k}\backslash G\bk{\A}} \varphi\bk{g}f_s\bk{\mu g}\, dg \ .
\]

Next we show that $I_\mu\bk{\varphi,f_s}=0$ unless $\mu$ is a representative of the open orbit.
\begin{enumerate}
\item
$\mu=1$. Then
\[
I_\mu\bk{\varphi,f_s}=\int\limits_{P\bk{k}\backslash G\bk{\A}} \varphi\bk{g}f_s\bk{g}dg=
\int\limits_{M\bk{k}U\bk{\A}\backslash G\bk{\A}} f_s\bk{g}\integral{U} \varphi\bk{ug}\, du\, dg=0 \ ,
\]
since $\varphi$ is cuspidal.

\item $\mu=w_2w_1, w_2w_3, w_2w_4.$ Then
\[
I_\mu\bk{\varphi,f_s}=\int\limits_{L\bk{k}R\bk{k}\backslash G\bk{\A}} \varphi\bk{g}f_s\bk{\mu g}\, dg=\int\limits_{L\bk{k}R\bk{\A}\backslash G\bk{\A}}f_s\bk{\mu g}\integral{R} \varphi\bk{rg}\, dr\, dg \ .
\]
Using \autoref{rallis_equation} this equals
\begin{align*}
&\int\limits_{L\bk{k}R\bk{\A}\backslash G\bk{\A}}f_s\bk{\mu g}
\sum_{\nu \in N_\beta\bk{k}\backslash L\bk{k}} W_\psi\bk{\varphi}\bk{\nu g}=\\ 
&\int\limits_{N_\beta\bk{\A}R\bk{\A}\backslash G\bk{\A}} f_s\bk{\mu g} W_\psi\bk{\varphi}\bk{g}
\bk{\integral{N_\beta} \psi\bk{n} \,dn} \,dg=0  \ .
\end{align*}
\end{enumerate}

Now let us compute the contribution from the open orbit.

For $\mu=w_2w_3 x_{-\alpha_1}\bk{1}$ it holds

\[
I_\mu\bk{\varphi,f}=\int\limits_{T^\mu\bk{k}U^\mu\bk{\A}\backslash G\bk{\A}}
\bk{\integral{U^\mu} \varphi\bk{ug}\, du}\, f_s(\mu g)\, dg \ .
\]
Expanding the function given by an inner integral along the root $\alpha+\beta$ and collapsing the sum with the outer integration the above equals
\begin{equation}
\label{eulerian_unfolding}
\int\limits_{U^\mu\bk{\A}\backslash G\bk{\A}} \integral{U} \varphi\bk{ug} \overline{\Psi_s\bk{u}}\, du\, f_s\bk{\mu g}\, dg \ .
\end{equation}

Since $U=U^\mu \cdot x_{\alpha+\beta}$ we bring the integral to its final form
\begin{equation}
\label{finite_eulerian_unfolding}
\int\limits_{U\bk{\A}\backslash G\bk{\A}} \integral{U} \varphi\bk{ug} \overline{\Psi_s \bk{u}}\, du\, 
\int\limits_{\A}f_s\bk{\mu x_{\alpha+\beta}\bk{r}g} \psi\bk{r}\, dr \, dg= \int\limits_{U\bk{\A}\backslash G\bk{\A}} L_{\Psi_s}\bk{\varphi}\bk{g} F\bk{g,s}\, dg \ .
\end{equation}
\end{proof}

\section[Derivation of The Main Theorem]{Derivation of The Main Theorem from Theorems \ref{thm:unramified} and \ref{thm:ramified}}
\label{PSRderive}
\begin{proof}[Proof of \autoref{thm:Main}]
By definition
\begin{equation}
\label{eq:Integral as a limit of integrals}
\zint\bk{s,\varphi,f}=
\mathop{\mathop{\lim_{\longrightarrow}}_{S\subset\Omega\subset{\Places}}}_{\Card{\Omega}<\infty}
\int_{U\bk{\A}_{\Omega}\setminus G\bk{\A}_{\Omega}} L_{\Psi_s}\bk{\varphi}\bk{g}F^\ast\bk{g,s}\, dg \ ,
\end{equation}
where $\displaystyle G\bk{\A}_{\Omega}=\prod_{\nu\in\Omega}G\bk{k_{\nu}}$. Fix $s_0\in\R$ such that the right hand side of \autoref{eq:unfolded integral} converges for $\Real{s}\more{s_0}$. The integrals of the right hand side of \autoref{eq:Integral as a limit of integrals} must also converge there. Also fix $s_1\in\R$ such that \autoref{unramified} holds for $\Real{s}\more{s_1}$. For a finite set $S\subseteq\Omega$ and $\nu\notin\Omega$ we have
\begin{align*}
&\int_{U\bk{\A}_{\Omega\cup\{\nu\}}\setminus G\bk{\A}_{\Omega\cup\{\nu\}}} L_{\Psi_s}\bk{\varphi}\bk{g}F^\ast\bk{g,s}\, dg = \\
=& \int_{U\bk{\A}_{\Omega}\setminus G\bk{\A}_{\Omega}} \int_{U\bk{k_{\nu}}\setminus G\bk{k_{\nu}}} L_{\Psi_s}\bk{\varphi}\bk{g g_{\nu}}F^\ast\bk{g g_{\nu},s}\, dg_{\nu}\, dg= \\
=& \int_{U\bk{\A}_{\Omega}\setminus G\bk{\A}_{\Omega}} F^\ast\bk{g,s} \int_{U\bk{k_{\nu}}\setminus G\bk{k_{\nu}}} L_{\Psi_s}\bk{\varphi}\bk{g g_{\nu}}F^\ast\bk{g_{\nu},s}\, dg_{\nu}\, dg= \\
=&\Lfun\bk{5s-2,\pi_{\nu},\st} 
\int_{U\bk{\A}_{\Omega}\setminus G\bk{\A}_{\Omega}} L_{\Psi_s}\bk{\varphi}\bk{g} F^\ast\bk{g,s}\, dg \ ,
\end{align*}
where the last equality is due to \autoref{thm:unramified}. A priori the last equality holds only
 for $\Real{s}\more\max\set{s_0,s_1}$, but since $\Lfun\bk{5s-2,\pi_{\nu},\st}$ is a meromorphic 
function the equality actually holds for $\Real{s}\more{s_0}$. Plugging this into 
\autoref{eq:Integral as a limit of integrals} we get
\begin{align*}
\zint\bk{s,\varphi,f}&=
\mathop{\mathop{\lim_{\longrightarrow}}_{S\subset\Omega\subset{\Places}}}_{\Card{\Omega}<\infty} 
\int_{U\bk{\A}_{\Omega}\setminus G\bk{\A}_{\Omega}} L_{\Psi_s}\bk{\varphi}\bk{g}F^\ast\bk{g,s}\, dg= \\
&= \mathop{\mathop{\lim_{\longrightarrow}}_{S\subset\Omega\subset{\Places}}}_{\Card{\Omega}<\infty} 
\prod_{\nu\in{\Omega\setminus{S}}} \Lfun\bk{5s-2,\pi_{\nu},\st} 
\int_{U\bk{\A}_{S}\setminus G\bk{\A}_{S}} L_{\Psi_s}\bk{\varphi}\bk{g}F^\ast\bk{g,s}\, dg=\\
&=\Lfun^{S}\bk{5s-2,\pi,\st} \int_{U\bk{\A}_{S}\setminus G\bk{\A}_{S}} 
L_{\Psi_s}\bk{\varphi}\bk{g}F^\ast\bk{g,s}\, dg \ .
\end{align*}

We finish the proof by fixing our datum according to \autoref{thm:ramified} and taking
\[
d_{S}\bk{s,{\varphi}_S,f_{S}}=\int_{U\bk{\A}_S\setminus G\bk{\A}_S} L_{\Psi_s}\bk{\varphi}\bk{g}F^\ast\bk{g,s}\, dg \ .
\]
\end{proof}

\section{Unramified Computation}
\label{sec:unramified}
Let $F=k_{\nu}$ with the ring of integers $\mathcal O$ and uniformizer $\varpi$  for some $\nu\notin{S}$. By abuse of notations we denote in this section, and in \autoref{App:Computing D-psi and convolution} and \autoref{App:Computing F(,s)}, $\pi$ for $\pi_{\nu}$, $\psi$ for $\psi_{\nu}$ etc. In this section we prove \autoref{thm:unramified}.
Recall that $G\bk{F}$ contains the maximal compact subgroup $K=G\bk{\mO}$. We fix on $G$ the Haar measure $\mu$ such that $\mu\bk{K}=1$. 

Recall that the Satake isomorphism is an isomorphism of $\C$-algebras $\Hecke \cong Rep\bk{{^L}G}$. Denote by $A_j\in\Hecke$ the elements corresponding to $Sym^j\bk{\st}$ by the Satake isomorphism. In particular for any unramified representation $\pi$ and a $K$-invariant vector $v^0\in\pi$ it holds
\begin{equation}
\label{Aj:definition}
\int_G A_j\bk{g} \pi\bk{g} v^0\, dg= tr\bk{Sym^j\bk{\st}}\bk{t_\pi} v^0,
\end{equation}
where $t_\pi$ is the Satake parameter of $\pi$.

For any unramified representation $\pi$ the Satake isomorphism induces an algebra homomorphism that sends  $f\in\Hecke$ to the complex number $\hat{f}\bk{\pi}$ such that
\[
\int_G f\bk{g} \pi\bk{g} v^0\, dg = \hat{f}\bk{\pi} v^0 \ .
\]
In particular for $f_1,f_2 \in \Hecke$ it holds $\widehat{f_1\ast f_2}=\hat{f_1} \cdot \hat{f_2}$. The homomorphism $f\rightarrow \hat f\bk{\pi}$ can be extended linearly to a map $\Hecke\coset{\coset{q^{-s}}} \rightarrow \C\coset{\coset{q^{-s}}}$.

\begin{Lem}[Poincar\' e identity]
There exists a generating function $\Delta\bk{g,s}\in\Hecke\coset{\coset{q^{-s}}}$ such that for any unramified representation $\pi$ with a spherical vector $v^0$ and {\bf any} functional $l$ on $\pi$ it holds
\begin{equation}
\label{eq:Poincare identity}
\int_{G} \Delta\bk{g,s} l\bk{\pi\bk{g}v^{0}}dg= \Lfun\bk{s,\pi,\st} l\bk{v^{0}} \,
\end{equation}
for $\Real{s}>>0$
\end{Lem}

\begin{proof}
We must show that there exists $\Delta$ with $\hat{\Delta} \bk{\pi,s}=\Lfun\bk{s,\pi,\st}$. The construction is formal.

\begin{align*}
\Lfun\bk{s,\pi,\st} &= \frac{1}{\det\bk{1-q^{-s}\st\bk{t_{\pi}}}}=
\prod_{i=1}^7 \bk{1-q^{-s}\st\bk{t_{\pi}}_{ii}}^{-1}= \\
&= \prod_{i=1}^7 \sum_{j=0}^{\infty} \bk{q^{-s}\st\bk{t_{\pi}}_{ii}}^{j} =
 \sum_{j=0}^{\infty} tr\bk{Sym^j\bk{\st\bk{t_{\pi}}}} q^{-js} \ ,
\end{align*}
where $t_{\pi}$ is the Satake parameter of $\pi$. The  series converge absolutely for $\Real{s}>>0$.
 Plugging \autoref{Aj:definition} into the previous equality gives
\[
\Lfun\bk{s,\pi,\st} l\bk{v^0}= 
\int_{G} \bk{\sum_{j=0}^{\infty} A_j  q^{-js}} \bk{g} l\bk{g\cdot v^0} dg \ .
\]
The assertion holds for $\Delta\bk{\cdot,s}=\sum_{j=0}^{\infty} A_j q^{-js}$ for any unramified representation $\pi$. Uniqueness follows from the fact that the action of the spherical functions of 
unramified representations gives rise to a spectral decomposition of $\Hecke$.
\end{proof}

For any $l\in \Hom_{U\bk{F}}\bk{\pi,\C_{\Psi_s}}$ one has
\[
\Lfun\bk{s,\pi,\st}\,l(v^0)=\int\limits_G l\bk{\pi(g)v^0} \Delta\bk{g,s}dg = \int\limits_{U\backslash G} l\bk{\pi\bk{g} v^0}\Delta^{\Psi_s}\bk{g,s}dg \ .
\]

Thus, in order to prove that for all unramified $\pi$ and all $l\in \Hom_U\bk{\pi,\Psi_s}$ \autoref{unramified} holds, it is enough to show the basic idenitity
\begin{equation}
\label{basic_identity}
\Delta^{\Psi_s}\bk{g,5s-2} = F^*\bk{g,s} \ .
\end{equation}
We prove this equality a priori only for $\Real{s}\more\more{0}$ but since for any $g\in{G}$ we know the right hand side to define a meromorphic function then this equality holds for all $s\in\C$. While the right hand side is given explicitly, we do not have an explicit formula for the generating function $\Delta\bk{g,s}$.

To overcome that difficulty we introduce the new function $D\in \Hecke\coset{\coset{q^{-s}}}$.
Recall the Cartan decomposition $G=KT^+K$ where  
$$T^{+}=\set{t\in{T}\mvert \FNorm{\gamma\bk{t}}\leq{1}\,\, \forall \gamma\in\Phi^{+}}.$$ 
The function $D$ is bi-$K$-invariant and is defined on the torus $T^+$ by
\[
D(t,s)=\FNorm{\omega_1(t)}^{5s+1} \qquad \forall t\in T^{+} \ 
\]
 The relation  between $D$ and $\Delta$ can be seen from the following proposition.

\begin{Prop} There exists $P\bk{g,s}\in \Hecke\coset{q^{-s}}$ and $s_0\in\R$ such that for $\Real{s}\more{s_0}$ it holds
\[
D\bk{\cdot,s}=\Delta\bk{\cdot, 5s-2}\ast P\bk{\cdot,s} \ .
\]
More precisely
\begin{equation}
P\bk{\cdot, s}= \frac{P_0\bk{q^{2-5s}}A_0-P_1\bk{q^{2-5s}}A_1}
{\zfun\bk{5s-1}\zfun\bk{5s+1}\zfun\bk{5s-2}},
\end{equation}
where
\[
P_0\bk{z}=
\frac{z^4}{q^2}+\bk{\frac{1}{q^2}+\frac{1}{q}}z^3+ \frac{z^2}{q}+\bk{\frac{1}{q}+1}z+1,
\quad 
P_1\bk{z}=\frac{z^2}{q} \ .
\]
\end{Prop}

\begin{proof}
Let $\omega_{\pi}$ be the normalized spherical function associated to $\pi$. For any functional $l$ of $\pi$ one has 
\[
\int_{G} D\bk{g,s}l\bk{\pi(g) v^0}\, dg=l\bk{v^0}\, \int_{G} D\bk{g,s} \omega_\pi\bk{g}\, dg \ .
\]
Using Macdonald's formula \cite[Theorem 4.2]{MR571057} for $\omega_\pi$ this integral turns to a sum of geometric progressions that converges for $\Real{s}\more\more{0}$. A direct computation yields
\begin{equation}
\label{eq:Macdonalds}
\hat{D}\bk{\pi,s}= \int_G D\bk{g,s}\omega_{\pi}\bk{g}\, dg= 
\Lfun\bk{5s-2,\pi,\st}\cdot Q\bk{s,\pi}\ .
\end{equation}
Here
\begin{equation}
Q\bk{\pi, s}=
\frac{P_0\bk{q^{2-5s}} -P_1\bk{q^{2-5s}} tr\bk{\st}\bk{t_\pi}}
{\zfun\bk{5s-1}\zfun\bk{5s+1}\zfun\bk{5s-2}}\ .
\end{equation}

On the other hand $\Lfun\bk{5s-2,\pi,st}=\hat\Delta\bk{\pi,5s-2}$ and obviously $Q\bk{\pi,s}=\hat{P}\bk{\pi,s}$. Since \autoref{eq:Macdonalds} holds for any unramified $\pi$ the proposition follows.

\end{proof}

Since the Fourier transform is a map of $\Hecke$ modules it follows:
\begin{Cor}
\[
D^{\Psi_s}\bk{s}= \Delta^{\Psi_s}\bk{5s-2} \ast P\bk{s} \ .
\]
\end{Cor}

The basic identity \ref{basic_identity} will follow once we prove
\begin{equation}
\label{new:basic:identity}
D^{\Psi_s}=F^\ast\conv P
\end{equation}
and 
\begin{Prop}
\label{invertible}
There exists $s_0$ such that whenever $\Real{s}>s_0$ 
 $f\ast P\bk{\cdot, s}=0$ implies  $f=0$
for any $f\in \mathcal{M}_{\Psi_s}$.
\end{Prop}

Indeed from \autoref{new:basic:identity} we get
\[
\bk{\Delta^{\Psi_s}\bk{\cdot, 5s-2} - F^\ast\bk{\cdot,s}} \conv P\bk{\cdot,s}=\bk{D^{\Psi_s}-F^\ast}\conv P=0
\]
and hence by \autoref{invertible} we have $\Delta^{\Psi_s}=F^\ast$ for $\Real{s}\more\more{0}$. As already mentioned this equality is actually true for all $s\in\C$. The only restriction on the convergence of the integral in \autoref{thm:unramified} is the domain of convergence of the Poincar\'e identity. We now turn to prove \autoref{invertible} and \autoref{new:basic:identity}.

The following observation is useful for the proof of \autoref{invertible}.
\begin{Remark}
We note that $\Hecke$ can be completed into a C$^{*}$-algebra $\hat{\Hecke}$ as a closed subspace of the reduced group C$^{*}$-algebra of $G$. One way to do this is to use the action of $\Hecke$ on $L^2\bk{K\setminus G/K}$ by convolution. This is a separable Hilbert space and $\Hecke$ admits an embedding into $\mathcal{B}\bk{L^2\bk{K\setminus G/K}}$ in which we complete it with respect to the operator norm. In fact, for our needs we only need to know that a C$^{*}$-norm and such a completion exist.
\end{Remark}

\begin{proof}[Proof of \autoref{invertible}]
We will show a stronger statement: there exists $s_0$ such that for any $\Real\bk{s}>s_0$ the element $P\bk{\cdot,s}$ is invertible in $\hat{\Hecke}$. For $\Real\bk{s}\more\more{0}$ this is equivalent to showing that 
\[
x:=A_0-\frac{P_1\bk{q^{2-5s}}}{P_0\bk{q^{2-5s}}}A_1
\]
is invertible. Since $\hat{\Hecke}$ is a C$^{*}$-algebra it will suffice to show that $\left\Vert \frac{P_1\bk{q^{2-5s}}}{P_0\bk{q^{2-5s}}}A_1\right\Vert\less{1}$.  We have
\[
\left\Vert \frac{P_1\bk{q^{2-5s}}}{P_0\bk{q^{2-5s}}}A_1\right\Vert=\FNorm{\frac{P_1\bk{q^{2-5s}}}{P_0\bk{q^{2-5s}}}} \left\Vert A_1\right\Vert
\]
and since 
\[
\lim_{\Real\bk{s}\rightarrow\infty} P_0\bk{q^{2-5s}}=1 \quad \text{and} \quad \lim_{\Real\bk{s}\rightarrow\infty} P_1\bk{q^{2-5s}}=0 \ ,
\]
there exists $s_0$ such that for $\Real\bk{s}\more{s_0}$ we have
\[
\FNorm{\frac{P_1\bk{q^{2-5s}}}{P_0\bk{q^{2-5s}}}} < \frac{1}{\left\Vert A_1\right\Vert} \ .
\]
\end{proof}

It remains to verify \autoref{new:basic:identity}. We shall evaluate
explicitly both functions and miraculously get the same answer.

\begin{Thm}
Both $D^{\Psi_s}$ and $F^\ast\conv P$ vanish outside $S_{\Psi_s} U T K$.
For $t=h_\alpha(t_1)h_\beta(t_2)$ such that $t_1,\frac{t_2}{t_1}\in\mO$ it holds  
\[
D^{\Psi_s}(t,s)=\bk{F^\ast \conv P}\bk{t,s}=
\piece{\frac{1+q^{1-5s}}{\zfun\bk{5s+1}} \FNorm{\frac{t_2}{t_1}}\FNorm{t_1}^{5s},& \FNorm{\frac{t_1^2}{t_2}}\less{1}\\ 
\frac{1+q^{1-5s}}{\zfun\bk{5s+1}}\FNorm{\frac{t_2}{t_1}}^{5s} \FNorm{t_1},& \FNorm{\frac{t_1^2}{t_2}}\more{1} \\
\frac{1+2q^{1-5s}}{\zfun\bk{5s+1}}\FNorm{t_1}^{5s+1},& \FNorm{\frac{t_1^2}{t_2}}={1}} \ .
\]
\end{Thm}

For the right hand side we first compute explicitly the function $F_s=\frac{F^\ast_s}{j\bk{s}}$ and then perform the convolution. This  tedious, but quite straightforward computation is performed in \autoref{App:Computing F(,s)}.

Now let us explain how to evaluate the left hand side. Let $SO_7$ be the 
special orthogonal group viewed it as a subgroup of $GL_7$, preserving the split symmetric form
$(\delta_{i,7-i})$.
Fix an embedding 
$\iota: G\bk{F}\rightarrow SO_7\bk{F}$ as in \cite{MR1617425}. 
 In \autoref{App:Computing D-psi and convolution} we give a realization of this map. 
Define a function $\Gamma: G\bk{F}\rightarrow \R$ by
\[
\Gamma\bk{g}=\max_{1\le i,j\le 7} \FNorm{\iota(g)_{i,j}} \ .
\]

The following result is easily checked.
\begin{Lem}
$\Gamma$ is a bi-$K$-invariant function and for $t\in T^{+}$
\[
\Gamma\bk{t}=\FNorm{\omega_1(t)}^{-1} \ .
\]
\end{Lem}

Thus $D\bk{g,s}=\sum_{k=0}^\infty D_k\bk{g}q^{-\bk{5s+1}k}$, where 
\[
D_k\bk{g}=
\piece{1, & \Gamma(g)=q^k \\ 0, & {\rm otherwise}} \ .
\]

For any $g\in{G}$ define $U_k\bk{g}=\set{u\in U: \Gamma(ug)\le q^k}$.
Then obviously 
\[
D^{\Psi_s}\bk{g,s}=\sum_{k=0}^{\infty} \bk{E_k\bk{g}-E_{k-1}\bk{g}} q^{-\bk{5s+1}k} \ ,
\]
where
\[
E_k\bk{g}= \int\limits_{U_k\bk{g}} \overline{\Psi_s\bk{u}} du, \ .
\]

The computation of $E_k\bk{g}$ can be further reduced to a calculation of volumes of certain sets. For a given $g$ there is at most two values of $k$ for which $E_k\bk{g}\neq 0$. The detailed computation is performed in \autoref{App:Computing D-psi and convolution}.

\section{Ramified Computation}
\label{ramified}
Fix a vector $\varphi=\otimes v_{\nu}$ in $\pi$ such that $L_{\Psi_s}\bk{\varphi}\not\equiv 0$ and $v_\nu$ is spherical outside of $S$. Recall from \autoref{thm: unfolding} that for the representative of the open orbit $\mu$
\[
d_S\bk{s,\varphi_S, f_s}= \int\limits_{U_S\backslash G_S} L_{\Psi_s}\bk{\varphi}\bk{g} F^\ast\bk{g,s}\, dg=j\bk{s} \int\limits_{U^\mu_S\backslash G_S}  L_\Psi\bk{\varphi}\bk{g} f\bk{\mu g,s}\, dg \ .
 \]

\subsection{Non-Archimedean case}
\begin{Prop}
\label{Prop:ramified non-arch}
Let $\nu$ be a finite prime. There exists $v\in \pi_\nu$ and a section $f\in Ind^{H\bk{k_\nu}}_{P_H\bk{k_v}}\delta_{P_H}^s$ such that for any $l\in \Hom_{U\bk{k_v}}\bk{\pi,\C_{\Psi_s}}$ and $s\in\C$ it holds
\[
\int\limits_{U^\mu_\nu\backslash G_\nu}  l\bk{\pi_\nu\bk{g}v} f\bk{\mu g,s} dg= l\bk{v} \ .
\]
\end{Prop}

\begin{proof}
Start with $v=v_{\nu}$, its  stabilizer contains a congruence compact subgroup  $K_m\subset K\bk{\mO_\nu}$ for some $m$.
Denote $\chi_s\bk{p}=\delta_{P_H}^s\bk{p^{\mu}}$. Since $\mu$ generates the open orbit, the map
\[
i:\Ind^H_{P_H} \delta_{P_H}^s\hookrightarrow \ind^G_{G^\mu} \chi_s \ ,
\]
defined  by $i(f)(g)=f(\mu g,s)$ is an isomorphism. By induction by stages the latter representation equals $\Ind^G_B \ind^B_{T^\mu U^\mu} \chi_s$.

Any $g\in G\bk{k_{\nu}}$ can be represented as $g^\mu h_{\beta}\bk{t}x_{\alpha}\bk{r_1}x_{\alpha+\beta}\bk{r_2}k$ 
with $g^\mu=u^\mu t^\mu \in G^\mu, k\in K$. We fix a section $f$ such that
\[
f\bk{g}=\chi_s\bk{g^\mu} \Id_{T^c\cap K_m}\bk{t} \Id_{U\cap K_m}\bk{r_1,r_2}\Id_{K_m}\bk{k} \ .
\]

Then
\[
\int\limits_{U^\mu_\nu\backslash G_\nu}  l\bk{\pi_\nu\bk{g}v} f\bk{\mu g,s} dg= \int\limits_{T^\mu_\nu} l\bk{\pi_\nu\bk{t^\mu} v} \chi_s\bk{t_\mu} dt_\mu= \int_{k^{\times}_\nu}l\bk{\pi\bk{h_{\omega_2}\bk{t}} \cdot v}\FNorm{t}^{5s} \, d^{\times}t \ .
\]

For any $\phi\in \Sch{k_v}$ define 
\[
\phi\ast v=\int_{k_v} \phi\bk{r}\pi\bk{x_{2\alpha+\beta}\bk{r}} v\, dr \ ,
\]
then 
\[
l\bk{h_{2\alpha+\beta}\bk{t} \phi\ast v}=\hat{\phi} \bk{t} l\bk{h_{2\alpha+\beta}\bk{t} v} \ .
\]
In particular, taking $\phi$ such that $\hat{\phi} =\Id_{1+\varpi^m \mathcal O}$ one has 
\[
\int\limits_{k_v^\times}l\bk{t\bk{\phi\ast v}} \FNorm{t}^{5s} \, dt= l\bk{v} \ .
\]
\end{proof}

Using the proposition above the computation of $d_S$ is reduced to the computation of $d_{\infty}$. Indeed, for a finite $\nu\in S$ denote $S'=S\backslash \set{\nu}$. Having $\varphi_S$ be a pure tensor vector $\otimes_{\nu'\in S'} v_{\nu'} \otimes v_\nu$ with $v_\nu$ as in \autoref{Prop:ramified non-arch}
\begin{align*}
d_S\bk{s,\varphi_S f_s} &= \int\limits_{U^\mu_{S}\backslash G_{S}} L_{\Psi_s}\bk{\varphi}\bk{g} f\bk{\mu g,s}\, dg= \\
&=\int\limits_{U^\mu_{S'}\backslash G_{S'}} \int\limits_{U^\mu_\nu\backslash G_\nu}   L_{\Psi_s}\bk{g_{S'} v\otimes g_\nu v_\nu}f_\nu\bk{\mu g_\nu,s} dg_\nu f_{S'}\bk{\mu g,s} dg= \\
&=\int\limits_{U^\mu_S\backslash G_S} L_{\Psi_s}\bk{g_{S'} v}f_{S'}\bk{\mu g_{s'},s} dg\  l\bk{v_\nu} \ .
\end{align*}

By induction the integral equals (up to a non-zero constant)
\[
d_{\infty}\bk{\varphi_{\infty}, f_{s,\infty},g}=  \int\limits_{U^\mu_{\infty}\backslash G_{\infty}} L_{\Psi_s}\bk{\varphi}\bk{g} f\bk{\mu g,s} dg \ .
\]

\subsection{Archimedean place}
\begin{Prop}
For any $s_0\in\C$ there exists $\varphi_{\infty}$ and $f_{\infty}$ so that $d_{S_\infty}(\varphi_{S_\infty}, f_{s_\infty},g)$ is analytic and non-zero in a neighbourhood of $s_0$.
\end{Prop}

\begin{proof}
Denote $\varphi_{\infty}=\otimes_{\nu|\infty}v_{\nu}$, arguing as in the non-Archimedean case it holds
\[
d_{\infty}\bk{s,\varphi_{\infty}, f_{s,\infty}}= \int\limits_{k_{\infty}^{\times}} L_{\Psi_s}\bk{h_{\omega_2}\bk{t} \varphi_{\infty}} \FNorm{t}^{5s}\, dt \ .
\]
Recall from \cite[2.6]{MR727854} that any Schwartz function 
$\phi\in\Sch{k_{\infty}}$ acts on $\pi_{\infty}$ by
\[
\phi\ast v=\int_{k_{\infty}} \phi\bk{r} \pi_{\infty}(x_{2\alpha+\beta}\bk{r} v)\, dr \ .
\]
As in the non-Archimedean case it holds
\[
d_{\infty}\bk{s,\phi\ast \varphi_{\infty}, f_{s,\infty}}= \int\limits_{k_{\infty}^{\times}}\hat{\phi}\bk{t} L_{\Psi_s}\bk{h_{\omega_2}\bk{t} \varphi_{\infty}} \FNorm{t}^{5s}\, dt \ .
\]
Recall that $L_{\Psi_s}\bk{\varphi}\neq{0}$. The function $\FNorm{t}^{5s}$ is $C^{\infty}$ in $t$ and analytic in $s$ on compact sets of the form $\set{t\mvert \FNorm{t-1}\less\epsilon}\times \set{s\mvert \FNorm{s-s_0}\less\epsilon}$. For $\phi\in\Sch{k_{\infty}}$ such that $\hat{\phi}$ is non-zero with compact support in $\set{t\mvert \FNorm{t-1}\less\epsilon}$ the function $d_{\infty}\bk{s,\phi\ast \varphi_{\infty}, f_{s,\infty}}$ is analytic in $\set{s\mvert \FNorm{s-s_0}\less\epsilon}$. One can choose $\epsilon$ and $\phi$ such that also
 $d_{\infty}\bk{s_0,\phi\ast \varphi_{\infty}, f_{s,\infty}}\neq 0$.
\end{proof}

\section{Application - $\Theta$-Lift For The Dual Pair $\bk{S_3,G_2}$}
The theta correspondence $\Theta_H$ for the dual pair $\bk{S_3, G_2}$ in the group $H\rtimes S_3$ has been studied in \cite{MR1918673}. The minimal representation $\Pi$ of $H$ can be extended to the group $H\rtimes S_3$. A cuspidal representation $\pi$ belongs to the image of $\Theta_H$ if 
\[
\integral{G}\varphi\bk{g} F\bk{g}dg\neq 0
\]
for $\varphi$ in the space $\pi$ and $F$ in the space of the minimal representation $\Pi$. It was proven in \cite{MR1918673} that any such representation $\pi$ supports the split Fourier coefficient. Besides, $\pi$ is a non-tempered representation and $\Lfun^S\bk{\pi,s,\st}$ has a double pole at $s=2$. Taking the residue (of depth $2$) at $s=2$ for the main equality we obtain the converse, i.e. the double pole of the standard $\Lfun$-function at $s=2$ characterizes the image of $\Theta_H$. In other words
\begin{Thm}
For a cuspidal representation $\pi$ of $G\bk{\A}$ that supports the split Fourier coefficient the following statements are equivalent
\begin{enumerate}
\item $\Lfun^S\bk{s,\pi,\st}$ has a double pole at $s=2$.
\item $\Theta_H\bk{\pi}\neq 0$.
\end{enumerate}
\end{Thm}

\appendix
\appendixpage
\section{Computing ${F\bk{\cdot,s}\conv P\bk{\cdot,s}}$}
\label{App:Computing F(,s)}
Recall that
\[
P(s)= P_0 A_0 + P_1 A_1
\]
and by \cite{MR1696481}
\[
A_0=\Id_{K},\quad
A_1=q^{-3}\bk{\Id_{K}+\Id_{K\omega_1\bk{\unif}K}} \ ,
\]
hence
\begin{equation}
\label{eq:Convolution formula}
F^\ast \bk{\cdot,s}\conv P\bk{\cdot,s}=
j\bk{s} \frac{\bk{P_0\bk{s}-q^{-3}P_1\bk{q^{-s}}} F\bk{\cdot,s}-q^{-3} P_1\bk{q^{2-5s}} F\bk{\cdot,s}\conv 1_{K\omega_1 K}\bk{\cdot}}
{\zfun\bk{5s-1}\zfun\bk{5s+1}\zfun\bk{5s-2}} \ .
\end{equation}
We shall compute each summand separately. In this section we prove the following result.
\begin{Prop}
\label{lemma:convolution}
The following holds for $\bk{F\bk{\cdot,s}\conv P\bk{\cdot,s}}$.
\begin{enumerate}
\item $\bk{F^\ast \bk{\cdot,s}\conv P\bk{\cdot,s}}\in\M_{\Psi_s}$.
\item $\bk{F^\ast \bk{\cdot,s}\conv P\bk{\cdot,s}}\bk{g}=0$ unless $g\in S_{\Psi_s}UTK$.
\item Let $t= h_{\alpha}\bk{t_1} h_{\beta}\bk{t_2}\in{T}$. If $t_1,\frac{t_2}{t_1}\in\mO$ it holds
\begin{equation}
\label{eq:value of convolution}
\bk{F^\ast \bk{\cdot,s}\conv P\bk{\cdot,s}}\bk{t}=\piece{\frac{1+q^{1-5s}}{\zfun\bk{5s+1}} \FNorm{\frac{t_2}{t_1}}\FNorm{t_1}^{5s},& \FNorm{\frac{t_1^2}{t_2}}\less{1}\\ 
\frac{1+q^{1-5s}}{\zfun\bk{5s+1}}\FNorm{\frac{t_2}{t_1}}^{5s} \FNorm{t_1},& \FNorm{\frac{t_1^2}{t_2}}\more{1} \\
\frac{1+2q^{1-5s}}{\zfun\bk{5s+1}}\FNorm{t_1}^{5s+1},& \FNorm{\frac{t_1^2}{t_2}}={1}} \ ,
\end{equation}
otherwise $\bk{F^\ast \bk{\cdot,s}\conv P\bk{\cdot,s}}\bk{t}=0$.
\end{enumerate}
\end{Prop}

\subsection{The spaces $\M_{\Psi_s}, \M^0_{\Psi_s}$}
In this subsection we list some properties of $\M_{\Psi_s}$ and $\M^0_{\Psi_s}$ which will be used in this section and in \autoref{App:Computing D-psi and convolution}. By Iwasawa decomposition any function $H\in \M_{\Psi_s}$ is determined by the values it attains on $B_M/(B_M\cap K)$, i.e. on the elements
\[
g=h_\alpha\bk{t_1} h_\beta\bk{t_2}x_{\alpha}\bk{d} \ ,
\]
where $d\in F/\mO$. In this appendix if $d\in \mO$ we choose a representative $d=1$.

Note that for any positive root $\gamma$ and $\FNorm{d}\geq{1}$ one has
\begin{align}
\label{eq:transpose with respect to cartan decompostion}
x_{-\gamma}(d)&=x_\gamma\bk{d^{-1}} h_{\gamma}\bk{d^{-1}} k \nonumber \\
x_{\gamma}(d)&=h_{\gamma}\bk{d} x_{-\gamma}\bk{d}  k' \ ,
\end{align}
for some $k,k'\in K$.

Using the invariance properties one easily checks the following lemma.
\begin{Lem} Let $g=h_\alpha\bk{t_1}h_\beta\bk{t_2}x_{\alpha}\bk{d}$.
\begin{enumerate}
\item Let  $H\in \M_{\Psi_s}$. Then $H\bk{g}=0$ unless
\begin{equation}
\label{eq:Condition for m without Spsi invariance}
t_1,\ dt_1+\frac{t_2}{t_1},\ 2d\frac{t_2}{t_1}+d^2t_1\in\mO \ .
\end{equation}
\item Let  $H\in \M^0_{\Psi_s}$. Then $H\bk{g}=0$ unless
\begin{equation}
\label{eq:Conditions for m}
t_1,\frac{t_2}{t_1},d^2t_1,d\frac{t_2}{t_1}\in\mO \ .
\end{equation}
\end{enumerate}
\end{Lem}

The following lemma will be useful in the computation of the second summand of \autoref{eq:Convolution formula}.

\begin{Lem}
\label{lemma:convolution of right K-invariant functions}
Let $H\in \M_{\Psi_s}$, $t\in T$  and $\Id_{KtK}$ be a characteristic function of the double coset $KtK$. Then 
\[
H\ast \Id_{KtK}\bk{g}=\sum_i H\bk{gb_i^{-1}} \ ,
\]
where $KtK=\coprod Kb_i$. Note that the representatives $b_i$ can be taken in the Borel subgroup $B$ of $G$.
\end{Lem}

\subsection{ Computation of $F_s$}
\begin{Prop}
\label{Lemma:Computing F(,s)} Assume that $g= h_{\alpha}\bk{t_1} h_{\beta}\bk{t_2}x_\alpha\bk{d}\in M$ satisfy \autoref{eq:Conditions for m}.
It holds
\[
F\bk{g,s}=\piece{
\frac{\zfun\bk{5s-1}}{\zfun\bk{5s}}\FNorm{t_1}^{5s}\FNorm{\frac{t_2}{t_1}}\bk{1-\FNorm{\unif\frac{t_2}{t_1}}^{5s-1}},& \FNorm{d^2\frac{t_1^2}{t_2}+d}\leq{1}\\
\frac{\zfun\bk{5s-1}}{\zfun\bk{5s}}\FNorm{t_1}^{5s}\FNorm{\frac{t_2}{t_1}}\FNorm{{d^2\frac{t_1^2}{t_2}+d}}^{1-5s}\bk{1-\FNorm{\unif\bk{{d^2\frac{t_1^2}{t_2}+d}}\frac{t_2}{t_1}}^{5s-1}},& \FNorm{d^2\frac{t_1^2}{t_2}+d}\geq{1} \ .
}
\]
For $g\in M$ violating \autoref{eq:Conditions for m} we have $F\bk{g,s}=0$.
\end{Prop}

\begin{proof}
We recall that
\[
F\bk{g,s}=\int\limits_{F} f_s\bk{\mu x_{\alpha+\beta}\bk{r} g} \psi\bk{r} dr \ ,
\]
where $f_s$ here the spherical section such that $f_s\bk{1}=1$.
For $g$ as above we have
\begin{align*}
F\bk{g,s}=&\int\limits_{F} f_s\bk{w_2w_3x_{-\alpha_1}\bk{1} x_{\alpha+\beta}\bk{r}  h_{\alpha}\bk{t_1} h_{\beta}\bk{t_2} x_{\alpha}\bk{d}} \psi\bk{r}  dr = \\
=&\FNorm{t_1}^{5s} \int\limits_{F} f_s\bk{w_2w_3x_{-\alpha_1}\bk{\frac{t_1^2}{t_2}} x_{\alpha_2+\alpha_3}\bk{\frac{t_1}{t_2} r} x_{\alpha_1}\bk{d}x_{\alpha_3}\bk{d}} \psi\bk{r} dr \ .
\end{align*}
Making a change of variables $r'=\frac{t_1}{t_2}$ and conjugating $w_3$ to the right we get
\begin{align*}
F\bk{g,s}=&\FNorm{t_1}^{5s}\FNorm{\frac{t_2}{t_1}} \int\limits_{F} f_s\bk{w_2 x_{-\alpha_1}\bk{\frac{t_1^2}{t_2}} x_{\alpha_2}\bk{r'} x_{\alpha_1}\bk{d}x_{-\alpha_3}\bk{d}} \psi\bk{\frac{t_2}{t_1} r'} dr' \ .
\end{align*}
Due to \autoref{eq:transpose with respect to cartan decompostion} we have
\begin{align*}
F\bk{m,s}=&\FNorm{t_1}^{5s}\FNorm{\frac{t_2}{t_1}} \int\limits_{F} f_s\bk{w_2 x_{-\alpha_1}\bk{\frac{t_1^2}{t_2}} x_{\alpha_2}\bk{r'} x_{\alpha_1}\bk{d}x_{\alpha_3}\bk{d^{-1}} h_{\alpha_3}\bk{d^{-1}}} \psi\bk{\frac{t_2}{t_1} r'} dr' \ .
\end{align*}
Conjugating the elements associated with $\alpha_3$ to the left and using a similar equality for $\alpha_1$ yields
\begin{align*}
F\bk{g,s}=&\FNorm{\frac{t_1}{d}}^{5s}\FNorm{\frac{t_2}{t_1}}  \int\limits_{F} f_s\bk{w_2 x_{-\alpha_1}\bk{\frac{t_1^2}{t_2}} x_{\alpha_2}\bk{\frac{r'}{d}} x_{\alpha_1}\bk{d}} \psi\bk{\frac{t_2}{t_1} r'} dr' = \\
=& \FNorm{\frac{t_1}{d}}^{5s}\FNorm{\frac{t_2}{t_1}}  \int\limits_{F} f_s\bk{w_2 x_{-\alpha_1}\bk{\frac{t_1^2}{t_2}} x_{\alpha_2}\bk{\frac{r'}{d}} h_{\alpha_1}\bk{d} x_{-\alpha_1}\bk{d}} \psi\bk{\frac{t_2}{t_1} r'} dr' = \\
=& \FNorm{t_1}^{5s}\FNorm{\frac{t_2}{t_1}}  \int\limits_{F} f_s\bk{w_2 x_{\alpha_2}\bk{r'} x_{-\alpha_1}\bk{d^2\frac{t_1^2}{t_2}+d} } \psi\bk{\frac{t_2}{t_1} r'} dr' \ .
\end{align*}
If $\FNorm{d^2\frac{t_1^2}{t_2}+d}\leq{1}$ we have
\begin{align*}
F\bk{g,s}=& \FNorm{t_1}^{5s}\FNorm{\frac{t_2}{t_1}}  \int\limits_{F} f_s\bk{w_2 x_{\alpha_2}\bk{r'}} \psi\bk{\frac{t_2}{t_1} r'} dr' \ .
\end{align*}

The integral is evaluated by separation to $\mO$ and $F\setminus\mO$ and once again using \autoref{eq:transpose with respect to cartan decompostion}. It holds
\begin{align*}
&\int\limits_{F} f_s\bk{w_2 x_{\alpha_2}\bk{r'}} \psi\bk{\frac{t_2}{t_1} r'} dr'= \int\limits_{\mO} f_s\bk{w_2 x_{\alpha_2}\bk{r'}} \psi\bk{\frac{t_2}{t_1} r'} dr' + \int\limits_{F\setminus\mO} f_s\bk{w_2 x_{\alpha_2}\bk{r'}} \psi\bk{\frac{t_2}{t_1} r'} dr' =\\
&= 1 + \int\limits_{F\setminus\mO} \FNorm{r'}^{-5s} \psi\bk{\frac{t_2}{t_1} r'} dr' = 1+\sum\limits_{1\less q^k\less \FNorm{\frac{t_2}{t_1}}} q^{-5ks} \int\limits_{\FNorm{r}=q^k} \psi\bk{r}dr -\FNorm{\frac{t_2}{t_1}}^{-5s} \int\limits_{\FNorm{r}=\FNorm{\frac{t_2}{t_1}}} \psi\bk{r}dr \\
&= \frac{\zfun\bk{5s-1}}{\zfun\bk{5s}} \bk{1-\FNorm{\unif\frac{t_2}{t_1}}^{5s-1}} \ .
\end{align*}
And hence
\[
F\bk{g,s}= \frac{\zfun\bk{5s-1}}{\zfun\bk{5s}}\FNorm{t_1}^{5s}\FNorm{\frac{t_2}{t_1}}\bk{1-\FNorm{\unif\frac{t_2}{t_1}}^{5s-1}} \ .
\]

Assume now that $\FNorm{d^2\frac{t_1^2}{t_2}+d}\more{1}$ and denote $p=d^2\frac{t_1^2}{t_2}+d$. It holds
\begin{align*}
F\bk{g,s}=& \FNorm{t_1}^{5s}\FNorm{\frac{t_2}{t_1}}  \int\limits_{F} f_s\bk{w_2 x_{\alpha_2}\bk{r'} x_{\alpha_1}\bk{p^{-1}} \check{\alpha_1}\bk{p^{-1}}} \psi\bk{\frac{t_2}{t_1} r'} dr' =\\
=& \FNorm{\frac{t_1}{p}}^{5s}\FNorm{\frac{t_2}{t_1}}  \int\limits_{F} f_s\bk{w_2 x_{\alpha_2}\bk{\frac{r'}{p}}} \psi\bk{\frac{t_2}{t_1} r'} dr'= \tag{$r''=\frac{r'}{p}$} \\
=& \FNorm{\frac{t_1}{p}}^{5s}\FNorm{p\frac{t_2}{t_1}}  \int\limits_{F} f_s\bk{w_2 x_{\alpha_2}\bk{r''}} \psi\bk{\frac{t_2}{t_1} pr''} dr'' \ .
\end{align*}
If $\FNorm{p\frac{t_2}{t_1}}\leq{1}$ then as in the previous case
\[
F\bk{g,s}=\frac{\zfun\bk{5s-1}}{\zfun\bk{5s}}\FNorm{t_1}^{5s}\FNorm{\frac{t_2}{t_1}}\FNorm{{d^2\frac{t_1^2}{t_2}+d}}^{1-5s}\bk{1-\FNorm{\unif\bk{{d^2\frac{t_1^2}{t_2}+d}}\frac{t_2}{t_1}}^{5s-1}} \ .
\]
If on the other hand $\FNorm{p\frac{t_2}{t_1}}\more{1}$ in a similar way to the previous cases it holds
\begin{align*}
&\int\limits_{F} f_s\bk{w_2 x_{\alpha_2}\bk{r}} \psi\bk{\frac{t_2}{t_1} pr} dr = \int\limits_{\mO} \psi\bk{\frac{t_2}{t_1} pr} dr + \int\limits_{F\setminus\mO} \FNorm{r}^{-5s} \psi\bk{\frac{t_2}{t_1} pr} dr =0+0  \ ,
\end{align*}
since $\psi$ is of conductor $\mO$. Note that when $\FNorm{p\frac{t_2}{t_1}}\more{1}$ \autoref{eq:Conditions for m} is violated.
\end{proof}

\begin{Remark}
Direct check shows that $F\bk{g,s}$ is $S_{\Psi_s}$-invariant and hence $F\bk{\cdot,s}\in \M^0_{\Psi_s}$.
\end{Remark}

\subsection{ Decomposition of $K\omega_1(\varpi)K$ into left $K$ cosets}
We recover the list of left $K$ cosets in $K\omega_1(\varpi)K$ from \cite[Proposition 13.3 and Proposition 14.2]{MR1932327}. The decomposition of $K\omega_1\bk{\varpi}K=\coprod b'_i K$ as a union of right $K$ cosets is given described there, after listing them we will make them into left cosets. Here $b'_i=n_i t_i$ belong to Borel subgroup $B=NT$. Fix $Y$ to be Teichm\" uller representatives in $\mO$ of $\mO/\bk{\unif}$ (or any other set of representatives) and $Z$ to be a set of representatives in $\mO$ of $\mO/\bk{\unif^2}$.
\begin{center}
\begin{tabular}{|c|c|c|}
\hline
Class $mod\ M\bk{\mO}$ & \# cosets & Representatives \\ 
\hline
$h_{-\omega_2}\bk{\unif}$ & $1$ & $h_{-\omega_2}\bk{\unif}$ \\ \hline
\multirow{2}{*}{$h_{-\beta}\bk{\unif}$} &
\multirow{2}{*}{$q^6$} &  $u\bk{r_1,r_2,r_3,r_4,r_5}h_{\omega_2}\bk{\unif}$ \\ 
 & &$r_1,r_2,r_3,r_4\in Y,\ r_5\in Z$ \\ \hline
\multirow{4}{*}{$h_{\alpha}\bk{\unif} h_{\beta}\bk{\unif}$} &
\multirow{4}{*}{$q\bk{q+1}$} & $u\bk{r_1,0,0,0,0}h_{-\alpha}\bk{\unif} h_{-\beta}\bk{\unif}$ \\
& & $r_1\in Y$ \\ \cline{3-3}
& & $u\bk{0,0,0,r_4,0}x_{\alpha}\bk{z}h_{-\beta}\bk{\unif}$\\
& & $r_4,z\in Y$ \\ \hline
\multirow{4}{*}{$h_{\omega_2}\bk{\unif}$} &
\multirow{4}{*}{$q^4\bk{q+1}$} & $u\bk{r_1,r_2,0,0,r_5} h_{\beta}\bk{\unif}$ \\ 
& & $r_2,r_5\in Y,\ r_1\in Z$ \\ \cline{3-3}
& & $u\bk{0,0,r_3,r_4,r_5}x_{\alpha}\bk{z}h_{\alpha}\bk{\unif} h_{\beta}\bk{\unif}$\\ 
& & $r_3,r_5,z\in Y,\ r_4\in Z$ \\ \hline
\multirow{6}{*}{$1$} &
\multirow{6}{*}{$q^3-1$} &
$u\bk{0,0,0,0,\frac{r_5}{\unif}}$ \\
& & $r_5\in Y,\ r_5\not\equiv{0}$ \\ \cline{3-3}
& & $u\bk{\frac{r_1}{\unif},0,0,0,\frac{r_5}{\unif}}$\\ 
& & $r_1, r_5\in Y,\ r_1\not\equiv{0}$\\ \cline{3-3}
& & $u\bk{\frac{y^3r_1}{\unif},\frac{y^2r_1}{\unif},\frac{yr_1}{\unif},\frac{r_1}{\unif},\frac{r_5}{\unif}}$\\ 
& & $r_1, r_5,y\in Y,\ r_1\not\equiv{0}$\\ \hline
\end{tabular}
\end{center}
We need now to make the right coset representatives $\set{b_i'}$ into left coset representatives. Let $w_0=w_\alpha w_\beta w_\alpha w_\beta w_\alpha w_\beta\in K$ be the longest element in the Weyl group of $G$. Recall that $w_0$ send $\gamma$ to $-\gamma$ for all $\gamma\in\Phi$. Also note that $\set{w_0b_i'}$ is also a full set of representatives of right cosets.

Denote by $\theta$ the Cartan antiinvolution, fixing the torus $T$ such that $\theta\bk{x_\gamma\bk{r}}=x_{-\gamma}\bk{r}$. Let $w_0\in K$ 
be a lifting to $G$ of the longest Weyl group element such that 
$\theta\bk{w_0}=w_0$. Then 
\begin{equation}
\label{eq:left coset representatives}
K\omega_1\bk{\varpi}K=\theta \bk{K\omega_1\bk{\varpi} K}=
\theta\bk{\coprod_i w_0 b_i' K}= \coprod_i K \theta\bk{b'_i} w_0=\coprod_i K  t_i^{-1} n_i \ .
\end{equation}
Fixing $b_i=t_i^{-1} n_i$ gives a set of left coset representatives.

\subsection{Convolution}
Combining \autoref{eq:Convolution formula}, \autoref{Lemma:Computing F(,s)} and \autoref{eq:left coset representatives} the computation of the convolution is straightforward. We shall present the computation for toral elements only. The case of non-toral elements is dealt similarly.

By \autoref{lemma:convolution of right K-invariant functions} we have
\[
\bk{F\bk{\cdot,s}\conv \Id_{K\lambda_1 K}}\bk{g}=\sum_{i}  F\bk{gb_i^{-1},s} \ .
\]

Assume that $g=t= h_{\alpha}\bk{t_1}  h_{\beta}\bk{t_2}$. By $S_{\Psi_s}$-invariance we may assume that $\FNorm{\frac{t_1^2}{t_2}}\leq{1}$. The case $\FNorm{\frac{t_1^2}{t_2}}\more{1}$ follows by symmetry from $\FNorm{\frac{t_1^2}{t_2}}\less{1}$, since $F\bk{\cdot,s}\in \M_{\Psi_s}$. We can write $\bk{F\bk{\cdot,s}\conv \Id_{K\lambda_1 K}}\bk{t}$ as follows
\begin{align}
\label{eq:convolution as a sum}
\bk{F\bk{\cdot,s}\conv \Id_{K\omega_1\bk{\unif}K}}\bk{t}=& 
F\bk{h_{-\omega_1}\bk{\unif}t,s}+q^6 F\bk{h_{\omega_1}\bk{\unif} t,s}+q F\bk{h_{-\alpha}\bk{\unif} h_{-\beta}\bk{\unif} t,s}+\nonumber \\
+& q\sum\limits_{y\in{\mO/\bk{\unif}}} F\bk{h_{-\beta}\bk{\unif} t\, x_{\alpha}\bk{-\frac{y}{\unif}},s}+ q^4 F\bk{h_{\beta}\bk{\unif} t,s}+\nonumber \\
+& q^4\sum\limits_{y\in{\mO/\bk{\unif}}} F\bk{h_{\alpha}\bk{\unif}h_{\beta}\bk{\unif} t\, x_{\alpha}\bk{-\frac{y}{\unif}},s}+\nonumber \\
+& \bk{q\sum\limits_{r,y\in{\mO/\bk{\unif}}}\psi\bk{-\frac{y}{\unif}\bk{\frac{t_2}{t_1}y+t_1}r}-1} F\bk{t,s}  \ .
\end{align}

We separate this computation into four cases depending on the absolute value of $t_1$ and $t_2$. All the following results follow by applying \autoref{Lemma:Computing F(,s)} to the summands in \autoref{eq:convolution as a sum}. Denote $\FNorm{t_1}=q^{-n}$ and $\FNorm{t_2}=q^{-m}$.

\begin{enumerate}
\item Assume $\FNorm{t_2}=\FNorm{t_1}=1$:
\[
\bk{F\bk{\cdot,s}*\Id_{K\lambda_1 K}}\bk{t}=\frac{q^{6-10s}+q^{5-5s}+3q^{4-5s}+2q^2-1}{\zfun\bk{5s}}
\]
and also
\[
F\bk{1,s}=\frac{1}{\zfun\bk{5s}} \ .
\]
Plugging this into \autoref{eq:Convolution formula} yields
\begin{align*}
\bk{F^\ast \bk{\cdot,s}\conv P\bk{\cdot,s}}\bk{t}=
\frac{1+2q^{1-5s}}{\zfun\bk{5s+1}} \ .
\end{align*}

\item Assume $\FNorm{t_2}\less\FNorm{t_1}$ and $\FNorm{\frac{t_1^2}{t_2}}=1$:
\begin{align*}
\bk{F\bk{\cdot,s}*\Id_{K\lambda_1 K}}\bk{t}=\frac{\zfun\bk{5s-1}}{\zfun\bk{5s}}&\left(\frac{q^{1+5s-n-5ns}}{\zfun\bk{5ns}}+\frac{q^{5-5s-n-5ns}}{\zfun\bk{5\bk{n+2}s}}+\frac{q^{2-n-5ns}}{\zfun\bk{5ns}}+\frac{q^{4-n-5(n+1)s}}{\zfun\bk{5\bk{n+1}s}}+\right. \\
&+\left(q^3-1\right)\frac{q^{-n-5ns}}{\zfun\bk{5\bk{n+1}s}}+
q\left(2\frac{q^{1-n-5ns}}{\zfun\bk{5ns}}+(q-2)\frac{q^{2-n-5(n+1)s}}{\zfun\bk{5\bk{n-1}s}}\right)+\\
&\left.+q^4\left(2\frac{q^{-n-5(n+1)s}}{\zfun\bk{5\bk{n+1}s}}+(q-2)\frac{q^{1-n-5(n+2)s}}{\zfun\bk{5ns}}\right)\right)
\end{align*}
and also
\[
F\bk{t,s}=\frac{\zfun\bk{5s-1}}{\zfun\bk{5s}}\frac{q^{-n-5ns}}{\zfun\bk{5\bk{n+1}s}} \ .
\]
Plugging this into \autoref{eq:Convolution formula} yields
\begin{align*}
\bk{F^\ast \bk{\cdot,s}\conv P\bk{\cdot,s}}\bk{t}= 
\frac{1+2q^{1-5s}}{\zfun\bk{5s+1}}\FNorm{t_1}^{5s+1} \ .
\end{align*}

\item Assume $\FNorm{t_2}\less\FNorm{t_1}$ and $\FNorm{\frac{t_1^2}{t_2}}\less 1$:
\begin{align*}
\bk{F\bk{\cdot,s}*\Id_{K\lambda_1 K}}\bk{t}=\frac{\zfun\bk{5s-1}}{\zfun\bk{5s}}&\left(\frac{q^{1+5s-m+n-5ns}}{\zfun\bk{5\bk{m-n}s}}+\frac{q^{5-5s-m+n-5ns}}{\zfun\bk{5\bk{m-n+2}s}}+\frac{q^{1-m+n-5ns+5s}}{\zfun\bk{5\bk{m-n+1}s}}+\right.\\
&+\frac{q^{3-5ns-m+n}}{\zfun\bk{5\bk{m-n+2}s}}+q\left(\frac{q^{1-5ns-m+n}}{\zfun\bk{5\bk{m-n}s}}+(q-1)\frac{q^{2-5s-m+n-5ns}}{\zfun\bk{5\bk{m-n-1}s}}\right)+\\
&\left.+q^4\left(\frac{q^{n-m-5s-5ns}}{\zfun\bk{5\bk{m-n+1}s}}+(q-1)\frac{q^{1-10s-m+n-5ns}}{\zfun\bk{5\bk{m-n}s}}\right)\right.+\\
&\left.+\left(q^3-1\right)\frac{q^{n-m-5ns}}{\zfun\bk{5\bk{m-n+1}s}}\right)
\end{align*}
and also
\[
F\bk{t,s}=\frac{\zfun\bk{5s-1}}{\zfun\bk{5s}}\frac{q^{n-m-5ns}}{\zfun\bk{5\bk{m-n+1}s}} \ .
\]
Plugging this into \autoref{eq:Convolution formula} yields
\begin{align*}
\bk{F^\ast \bk{\cdot,s}\conv P\bk{\cdot,s}}\bk{t}= 
\frac{1+q^{1-5s}}{\zfun\bk{5s+1}} \FNorm{\frac{t_2}{t_1}}\FNorm{t_1}^{5s} \ .
\end{align*}

\item Assume $\FNorm{t_2}=\FNorm{t_1}$ and $\FNorm{\frac{t_1^2}{t_2}}\less 1$:
\begin{align*}
\bk{F\bk{\cdot,s}*\Id_{K\lambda_1 K}}\bk{t}=\frac{\zfun\bk{5s-1}}{\zfun\bk{5s}}&\left(\frac{q^{5-5s-5ns}}{\zfun\bk{10s-2}}+\frac{q^{1+5s-5ns}}{\zfun\bk{5s-1}}+\frac{q^{3-5ns}}{\zfun\bk{10s-2}}+\right.\\
&\left.+\frac{q^{4-5s-5ns}}{\zfun\bk{5s-1}}+\left(q^2-1\right)\frac{q^{-5ns}}{\zfun\bk{5s-1}}\right)
\end{align*}
and also
\[
F\bk{t,s}=\frac{\zfun\bk{5s-1}}{\zfun\bk{5s}}\frac{q^{-5ns}}{\zfun\bk{5s-1}} \ .
\]
Plugging this into \autoref{eq:Convolution formula} yields
\begin{align*}
\bk{F^\ast \bk{\cdot,s}\conv P\bk{\cdot,s}}\bk{t}= 
\frac{1+q^{1-5s}}{\zfun\bk{5s+1}}\FNorm{t_1}^{5s} \ .
\end{align*}

\end{enumerate}

\section{Computation of $D_{\Psi_s}$}
\label{App:Computing D-psi and convolution}
Recall that our aim is to compute
\[
E_k\bk{g}=\int_{U_k\bk{g}} \overline{\Psi_s\bk{u}} du \ .
\]
We treat first the case where $g\in S_{\Psi_s} UTK$ and then the case where $g\notin S_{\Psi_s} UTK$.

We note the following helpful fact that will be used repeatedly through out this section.
\begin{Lem}
\label{lemma:measuring lemma}
For $a,b,c\in\N$ with $a+b\ge c$ it holds
\[
\meas\set{\bk{x,y}\mvert \FNorm{x}\leq{q^a},\quad \FNorm{y}\leq{q^b},\quad \FNorm{xy}\leq{q^c}}= q^{c}\bk{1+\bk{a+b-c}\bk{1-q^{-1}}} \ .
\]
\end{Lem}

\subsection{ Toral elements } 
For $t=h_\alpha(t_1)h_\beta(t_2)$ and $u=u(r_1,r_2,r_3,r_4,r_5)$, the matrix $\iota(ut)$ has form 
\begin{equation}
\begin{pmatrix}
1 & 0 & r_2 & r_3 & \frac{-r_4}{2} & \frac{r_2r_3+r_5}{2} & \frac{r_2r_4-r_3^2}{2} \\
0 & 1 & r_1 & r_2& \frac{-r_3}{2} & \frac{r_1r_3-r_2^2}{2} & \frac{r_1r_4-2r_2r_3-r_5}{2} \\
0 & 0 & 1 & 0 & 0 & \frac{r_3}{2} & \frac{r_4}{2} \\
0 & 0 & 0 & 1 & 0 & -r_2 & -r_3 \\
0 & 0 & 0 & 0 & 1 & -r_1 & -r_2 \\
0 & 0 & 0 & 0 & 0 & 1 & 0 \\
0 & 0 & 0 & 0 & 0 & 0 & 1
\end{pmatrix}
\cdot
\begin{pmatrix}
t_1\\
&\frac{t_2}{t_1}\\
&&\frac{t_1^2}{t_2}\\
&&&1\\
&&&&\frac{t_2}{t_1^2}\\
&&&&&\frac{t_1}{t_2}\\
&&&&&&\frac{1}{t_1}
\end{pmatrix} \ .
\end{equation}
Consider an element $g\in T$ and denote it by $t=h_\alpha(t_1)h_\beta(t_2)$. 
Denote $\FNorm{t_1}=q^{-n}$ and $\FNorm{t_2}=q^{-m}$.
 By \autoref{eq:Conditions for m} $E_k\bk{t}=0$ unless $\FNorm{t_1}, \FNorm{\frac{t_2}{t_1}}\le 1$.
 Since $E_k\in \M^0_{\Psi_s}$ and $w_\alpha \in S_{\Psi_s}\subset K$ one has 
$$E_k\bk{t}=E_k\bk{w_\alpha t w_\alpha^{-1}}.$$ 
In particular we can assume  $\FNorm{\alpha\bk{t}}=\FNorm{\frac{t_1^2}{t_2}}\leq{1}$ or 
$|t_1|\le |\frac{t_2}{t_1}|$.  Also, $U_k\bk{t}=\emptyset$ unless $|t_1|\ge q^{-k}$.
To sum up we have to compute $E_k(t)$ where
\[
q^{-k}\leq \FNorm{t_1}\leq\FNorm{\frac{t_2}{t_1}}\leq 1 \ .
\]
We may exchange integration over $U_k\bk{t}$ to integration over a smaller and simpler set, namely
\begin{Lem}
\label{lemma:reducing domain of integration}
\[
E_k\bk{t}=\int\limits_{\overline{U_k\bk{t}}} \overline{\Psi_s\bk{u}} \, du \ ,
\]
where
\[
\overline{U_k\bk{t}}=\set{u\bk{r_1,r_2,r_3,r_4,r_5}\in U_k\bk{t}\mvert \FNorm{r_2},\FNorm{r_3}\leq q} \ .
\]
\end{Lem}

\begin{proof}
For any $x,y \in F$ define 
\[
U^{\bk{x,y}}_k\bk{t}=\set{u\bk{r_1,r_2,r_3,r_4,r_5}\in U_k\bk{t}\mvert r_2=x,r_3=y}
\]
and note that for $s_1,s_2 \in \mO^\times$
\[
U^{\bk{s_1 x,s_2 y}}_k\bk{t}= h\bk{s_1,s_2}U^{\bk{x,y}}_k\bk{t} h^{-1}\bk{s_1,s_2}\ ,
\]
where $h\bk{s_1,s_2}=h_\beta\bk{s_1}h_{2\alpha+\beta}\bk{s_2}$. Since $\delta_P\bk{h\bk{s_1,s_2}}=1$ it follows that $\meas\bk{U^{\bk{s_1x,s_2y}}_k\bk{t}}=\meas\bk{U^{\bk{x,y}}_k\bk{t}}$ which means it depends only on $t$, $\FNorm{x}$ and $\FNorm{y}$. In particular if $\FNorm{x}=q^i, \FNorm{y}=q^j$ we denote $\meas\bk{U^{\bk{x,y}}_k\bk{t}}$ by $\meas \bk{U^{i,j}_k\bk{t}}$.

Thus
\begin{align*}
E_k\bk{t}&=\int\limits_{\overline{U_k\bk{t}}} \overline{\Psi_s\bk{u}} \, du 
= \int\limits_{F\times F} \meas\bk{U^{\bk{x,y}}_k\bk{t}} \overline{\psi\bk{x+y}}\, dx\, dy=\\
&=\sum^\infty_{i,j=-\infty}  \meas\bk{U^{i,j}_k\bk{t}}
\int\limits_{\FNorm{x}=q^{i}}\overline{\psi\bk{x}}\, dx \int\limits_{\FNorm{y}=q^{j}} \overline{\psi\bk{y}}\, dy \ .
\end{align*}
Since $\int\limits_{\FNorm{z}=q^k}\overline{\psi(z)}dz=0$ for $k\more 1$, the proposition follows.
\end{proof}

\begin{Remark}
We can describe $\overline{U_k\bk{t}}$ by a short list of inequalities. Namely $u\in\overline{U_k\bk{t}}$ if and only if
\begin{align*}
&k\geq{n} \\
&\FNorm{r_2}, \FNorm{r_3}\leq{q}\\
&\FNorm{r_1},\FNorm{r_2},\FNorm{r_3}, \FNorm{r_2r_3+r_5},\FNorm{r_1r_3-r_2^2} \leq{q^{k+n-m}}\\
&\FNorm{r_2},\FNorm{r_3},\FNorm{r_4}, \FNorm{r_2r_4-r_3^2},\FNorm{r_1r_4-2r_2r_3-r_5}\leq q^{k-n} \ .
\end{align*}
\end{Remark}

\begin{Cor}
\label{cor:E_k as sum of four volumes}
\[
E_k\bk{t}=\meas\bk{U_k^{0,0}\bk{t}}-\meas\bk{U_k^{0,1}\bk{t}} -\meas\bk{U_k^{1,0}\bk{t}}+\meas\bk{U_k^{1,1}\bk{t}} \ .
\]
\end{Cor}

\begin{proof}
We first note that for every $i,i'\leq{0}$, $j\leq 1$ and $k$ it holds
\[
\meas\bk{U^{i,j}_k\bk{t}}=\meas\bk{U^{i',j}_k\bk{t}}, \quad
\meas\bk{U^{j,i}_k\bk{t}}=\meas\bk{U^{j,i'}_k\bk{t}} \ .
\]
We also recall that
\[
\int\limits_{\FNorm{r}\leq 1} \overline{\psi\bk{r}}=1, \quad \int\limits_{\FNorm{r}=q} \overline{\psi\bk{r}}=-1 \ .
\]
The claim then follows by a simple computation
\begin{align*}
E_k\bk{t}&=\sum^\infty_{i,j=-\infty}  \meas\bk{U^{i,j}_k\bk{t}}
\int\limits_{\FNorm{x}=q^{i}} \overline{\psi\bk{x}}\, dx \int\limits_{\FNorm{y}=q^{j}} \overline{\psi\bk{y}}\, dy= \\
&=\meas\bk{U_k^{0,0}\bk{t}} \int\limits_{\FNorm{x}\leq{1}} \overline{\psi\bk{x}}\, dx \int\limits_{\FNorm{y}\leq{1}} \overline{\psi\bk{y}}\, dy -
\meas\bk{U_k^{0,1}\bk{t}} \int\limits_{\FNorm{x}\leq{1}} \overline{\psi\bk{x}}\, dx \int\limits_{\FNorm{y}=q} \overline{\psi\bk{y}}\, dy - \\
&-\meas\bk{U_k^{1,0}\bk{t}} \int\limits_{\FNorm{x}=q} \overline{\psi\bk{x}}\, dx \int\limits_{\FNorm{y}\leq{1}} \overline{\psi\bk{y}}\, dy
+\meas\bk{U_k^{1,1}\bk{t}} \int\limits_{\FNorm{x}=q} \overline{\psi\bk{x}}\, dx \int\limits_{\FNorm{y}=q} \overline{\psi\bk{y}}\, dy=\\
&=\meas\bk{U_k^{0,0}\bk{t}}-\meas\bk{U_k^{0,1}\bk{t}} -\meas\bk{U_k^{1,0}\bk{t}}+\meas\bk{U_k^{1,1}\bk{t}} \ .
\end{align*}
\end{proof}

\begin{Prop}
For $t$ as above, with $\FNorm{t_1}=q^{-n}$, it holds
\begin{enumerate}
\item $E_k\bk{t}=0$ for $k\neq n, n+1$.
\item $E_n\bk{t}=\piece{1 & \FNorm{\alpha\bk{t}}=1\\ \FNorm{\alpha\bk{t}}^{-1} & \FNorm{\alpha\bk{t}}\less 1}, \quad
E_{n+1}\bk{t}=\piece{2q^2 & \FNorm{\alpha\bk{t}}=1\\ 2q^2 \FNorm{\alpha\bk{t}}^{-1} & \FNorm{\alpha\bk{t}}\less 1}$.
\end{enumerate}
\end{Prop}

\begin{proof}
We separate the proof according to the absolute value of $\alpha\bk{t}$.
\begin{itemize}
\item Assume that $\FNorm{\alpha\bk{t}}=1$, i.e. $\FNorm{\frac{t_1^2}{t_2}}=1$.
\begin{enumerate}
\item Assume $k=n$, then $u\in\overline{U_k\bk{t}}$ if and only if
\[
\FNorm{r_1},\FNorm{r_2},\FNorm{r_3},\FNorm{r_4},\FNorm{r_5}\leq{1} \ .
\]
In this case $U_k^{0,1}\bk{t},U_k^{1,0}\bk{t},U_k^{1,1}\bk{t}=\emptyset$ and $U_k^{0,0}\bk{t}=\mO^3$.
Hence
\[
E_n\bk{g}=1
\]

\item Assume $k=n+1$, then $u\in\overline{U_k\bk{t}}$ if and only if
\begin{align*}
&\FNorm{r_1},\FNorm{r_2},\FNorm{r_3},\FNorm{r_4} \leq{q}\\
&\FNorm{r_2r_3+r_5},\FNorm{r_1r_3-r_2^2}, \FNorm{r_2r_4-r_3^2},\FNorm{r_1r_4-2r_2r_3-r_5}\leq q \ .
\end{align*}
We demonstrate the measurement of $U_k^{i,j}\bk{t}$ in this case as an example to the calculation held in all other cases. 
Assume that $\FNorm{r_2},\FNorm{r_3}\leq{1}$, then $u\in\overline{U_k\bk{t}}$ if and only if
\[
\FNorm{r_1},\FNorm{r_4},\FNorm{r_5}, \FNorm{r_1r_4} \leq{q}\\
\]
and hence, by \autoref{lemma:measuring lemma},
\[
\meas\bk{U_k^{0,0}\bk{t}}=q\bk{q+\bk{q-1}}=2q^2-q \ .
\]

Assume $\FNorm{r_2}=q$ and $\FNorm{r_3}\leq{1}$. Then $\FNorm{r_1r_3}\leq{q}$ but also $\FNorm{r_1r_3-r_2^2}\leq{q}$ which contradicts the fact that $\FNorm{r_2^2}=q^2$. Hence $U_k^{0,1}\bk{t}=\emptyset$, by a similar argument $U_k^{1,0}\bk{t}=\emptyset$.

Assume that $\FNorm{r_2},\FNorm{r_3}=q$. Let us parametrize $U_k^{1,1}\bk{t}$ in the following way
\begin{align*}
r_1&=\frac{x+r_2^2}{r_3} \\
r_4&=\frac{y+r_3^2}{r_2} \\
r_5&=z-r_2r_3 \ .
\end{align*}
The domain of integration for the new variables is $\FNorm{x},\FNorm{y},\FNorm{z}\leq{q}$. Also
\[
dr_1=\frac{dx}{q},\quad dr_4=\frac{dy}{q},\quad dr_5=dz \ .
\]
Note that now
\[
\FNorm{r_1r_4-2r_2r_3-r_5}=\FNorm{\frac{x+r_2^2}{r_3}\cdot\frac{y+r_3^2}{r_2}-r_2r_3-z} = 
\FNorm{\frac{xy+xr_3^2+yr_2^2}{r_2r_3}-z}\leq{q} \ .
\]
Hence
\[
\meas\bk{U_k^{1,1}\bk{t}}=\int\limits_{\unif^{-1}\mO} \frac{dx}{q} \int\limits_{\unif^{-1}\mO} \frac{dy}{q} \int\limits_{\unif^{-1}\mO} dz=q \ .
\]

Combining the computed $\meas\bk{U^{i,j}_k\bk{t}}$ yields
\begin{align*}
E_{n+1}\bk{t}=& \meas\bk{U_k^{0,0}\bk{t}}-\meas\bk{U_k^{0,1}\bk{t}}- \meas\bk{U_k^{1,0}\bk{t}}-\meas\bk{U_k^{1,1}\bk{t}}=\\
=& \bk{2q^2-q}-0-0+q=2q^2 \ .
\end{align*}

\item Assume $k\more{n+1}$, then $u\in\overline{U_k\bk{t}}$ if and only if
\begin{align*}
&\FNorm{r_2}, \FNorm{r_3}\leq{q}\\
&\FNorm{r_1},\FNorm{r_4},\FNorm{r_5}, \FNorm{r_2r_4},\FNorm{r_1r_3}, \FNorm{r_1r_4}\leq q^{k-n} \ .
\end{align*}
Hence, according to \autoref{lemma:measuring lemma},
\begin{align*}
\meas\bk{U_k^{0,0}\bk{t}}&=q^{2\bk{k-n}}\bk{1+\bk{k-n}\bk{1-q^{-1}}} \\
\meas\bk{U_k^{1,0}\bk{t}}&=\meas\bk{U_k^{0,1}\bk{t}}=q^{2\bk{k-n}}\bk{1+\bk{k-n-1}\bk{1-q^{-1}}} \\
\meas\bk{U_k^{1,1}\bk{t}}&=q^{2\bk{k-n}}\bk{1+\bk{k-n-2}\bk{1-q^{-1}}} \ ,
\end{align*}
and then
\begin{align*}
E_k\bk{g}=0 \ .
\end{align*}


Evaluating $D^{\Psi_s}$ at $t$ yields
\[
D^{\Psi_s}\bk{t}= q^{-n}E_n\bk{t} +q^{-n-1}\bk{ E_{n+1}\bk{t}-E_{n}\bk{t}}+ q^{-n-2}E_{n+1}\bk{t}= \frac{1+2q^{1-5s}}{\zfun\bk{5s+1}}\FNorm{t_1}^{5s+1} \ .
\]
\end{enumerate}

\item Assume that $\FNorm{\alpha\bk{t}}\less 1$, i.e. $\FNorm{\frac{t_1^2}{t_2}}\less{1}$.
\begin{enumerate}
\item Assume $k=n$, then $u\in\overline{U_k\bk{t}}$ if and only if
\begin{align*}
&\FNorm{r_2},\FNorm{r_3},\FNorm{r_4},\FNorm{r_1r_4-r_5}\leq 1\\
&\FNorm{r_1}, \FNorm{r_5} \leq{q^{2n-m}} \ .
\end{align*}
By making a change of variables $r_5=x+r_1r_4$ this is equivalent to
\begin{align*}
&\FNorm{r_2},\FNorm{r_3},\FNorm{r_4},\FNorm{x}\leq 1\\
&\FNorm{r_1}, \FNorm{r_1r_4} \leq{q^{2n-m}} \ .
\end{align*}
Hence, according to \autoref{lemma:measuring lemma},
\begin{align*}
\meas\bk{U_k^{0,0}\bk{t}}&=q^{2n-m} \\
\meas\bk{U_k^{1,0}\bk{t}}&= \meas\bk{U_k^{0,1}\bk{t}}= \meas\bk{U_k^{1,1}\bk{t}}=0 \ ,
\end{align*}
and then
\[
E_n\bk{t}=q^{2n-m} \ .
\]

\item Assume $k=n+1$, then $u\in\overline{U_k\bk{t}}$ if and only if
\begin{align*}
&\FNorm{r_2}, \FNorm{r_3}, \FNorm{r_4}, \FNorm{r_2r_4-r_3^2},\FNorm{r_1r_4-2r_2r_3-r_5}\leq{q}\\
&\FNorm{r_1},\FNorm{r_5},\FNorm{r_1r_3} \leq{q^{2n-m+1}} \ .
\end{align*}
Hence, according to \autoref{lemma:measuring lemma} and arguments similar to case 2 with $\FNorm{\alpha\bk{t}}=1$,
\begin{align*}
&\meas\bk{U_k^{0,0}\bk{t}}=q^{2n-m+2}\bk{1+\bk{1-q^{-1}}}, &\meas\bk{U_k^{1,0}\bk{t}}=q^{2n-m+2} \\ &\meas\bk{U_k^{0,1}\bk{t}}= 0, &\meas\bk{U_k^{1,1}\bk{t}}=q^{2n-m+1} \ ,
\end{align*}
and then
\[
E_n\bk{t}=q^{2n-m+2} \ .
\]

\item Assume $k\more{n+1}$, then $u\in\overline{U_k\bk{g}}$ if and only if
\begin{align*}
&\FNorm{r_2}, \FNorm{r_3}\leq{q}\\
&\FNorm{r_1}, \FNorm{r_5},\FNorm{r_1r_3} \leq{q^{k+n-m}}\\
&\FNorm{r_4}, \FNorm{r_2r_4},\FNorm{r_1r_4-r_5}\leq q^{k-n}.
\end{align*}
By making a change of variables $r_5=x+r_1r_4$ this is equivalent to
\begin{align*}
&\FNorm{r_2}, \FNorm{r_3}\leq{q}\\
&\FNorm{r_1}, \FNorm{r_1r_4},\FNorm{r_1r_3} \leq{q^{k+n-m}}\\
&\FNorm{r_4}, \FNorm{r_2r_4},\FNorm{x}\leq q^{k-n}.
\end{align*}
Hence, according to \autoref{lemma:measuring lemma},
\begin{align*}
\meas\bk{U_k^{0,0}\bk{t}}&=q^{k-n}q^{k+n-m}\bk{1+\bk{k-n}\bk{1-q^{-1}}} \\
\meas\bk{U_k^{1,0}\bk{t}}&=q^{k-n}q^{k+n-m}\bk{1+\bk{k-n-1}\bk{1-q^{-1}}} \\
\meas\bk{U_k^{0,1}\bk{t}}&=q^{k-n}q^{k+n-m}\bk{1+\bk{k-n-1}\bk{1-q^{-1}}} \\
\meas\bk{U_k^{1,1}\bk{t}}&=q^{k-n}q^{k+n-m}\bk{1+\bk{k-n-2}\bk{1-q^{-1}}} \ ,
\end{align*}
and then
\[
E_n\bk{t}=0 \ .
\]

Evaluating $D^{\Psi_s}$ at $t$ yields
\[
D^{\Psi_s}\bk{t}= \frac{1+q^{1-5s}}{\zfun\bk{5s+1}} \FNorm{\frac{t_2}{t_1}}\FNorm{t_1}^{5s} \ .
\]
\end{enumerate}
\end{itemize}
\end{proof}

\subsection{Non-toral case}
This case is  technically more involved than the case of the toral elements, but all the ideas for the toral elements can be carried to this case as well. We will prove the following result.
\begin{Prop}
\label{Prop:non toral Fourier coefficient}
$E_k(g)=0$ for $g\notin S_{\Psi_s} UTK$.
\end{Prop}

Let $g=t x_\alpha(d)$, where $t=h_{\alpha}\bk{t_1} h_{\beta}\bk{t_2}$ and $\FNorm{d}>1$. Since $g\notin S_{\Psi_s} UTK$ it holds $\FNorm{d^2\alpha\bk{t}+d}\geq 1$.
By \autoref{eq:Conditions for m} $E_k\bk{t}=0$ unless
\[
t_1,\frac{t_2}{t_1},d^2t_1,d\frac{t_2}{t_1}\in\mO \ .
\]
Since $E_k\in \M^0_{\Psi_s}$ and $w_\alpha \in S_{\Psi_s}\subset K$ one has $E_k\bk{g}=E_k\bk{w_\alpha g w_\alpha^{-1}}$. Hence it is enough to compute $E_k\bk{g}$ when $\FNorm{d\alpha\bk{t}}=\FNorm{d\frac{t_1^2}{t_2}}\leq 1$.

The matrix $\iota\bk{x_{\alpha}\bk{d}}$ has the form
\[
\begin{pmatrix}
1 & d & 0 & 0 & 0 & 0 & 0 \\
0 & 1 & 0 & 0 & 0 & 0 & 0 \\
0 & 0 & 1 & -d & -\frac{d^2}{2} & 0 & 0 \\
0 & 0 & 0 & 1 & d & 0 & 0 \\
0 & 0 & 0 & 0 & 1 & 0 & 0 \\
0 & 0 & 0 & 0 & 0 & 1 & -d \\
0 & 0 & 0 & 0 & 0 & 0 & 1
\end{pmatrix} \ .
\]
Denote $\FNorm{t_1}=q^{-n}$, $\FNorm{t_2}=q^{-m}$ and $\FNorm{d\alpha\bk{t}}=q^l$. Under this notations $U_k\bk{g}=\emptyset$ when $k\less{n}$ and so we may assume that $k\geq{n}$.

We now reduce the domain of integration. The proof of this lemma is similar to the proof of \autoref{lemma:reducing domain of integration} and is omitted.
\begin{Lem}
\[
E_k\bk{t}=\int\limits_{\widehat{U_k\bk{g}}} \overline{\Psi_s\bk{u}} du \ ,
\]
where
\[
\widehat{U_k\bk{g}}=\set{ u(r_1, r_2, r_3, r_4, r_5)\in U_k\bk{g}\mvert \FNorm{r_2+r_3}\leq q} \ .
\]
\end{Lem}

\begin{Remark}
Denote $b=d\frac{t_1^2}{t_2}$. When $\FNorm{d\alpha\bk{t}}\leq{1}$, we have $u\in\widehat{U_k\bk{g}}$ if and only if
\begin{align*}
&k\geq n\\
&\FNorm{r_2+r_3}\leq{q}\\
&\FNorm{r_1},\FNorm{r_2},\FNorm{r_3}, \FNorm{r_2r_3+r_5},\FNorm{r_1r_3-r_2^2} \leq{q^{k+n-m}}\\
&\FNorm{br_1-r_2},\FNorm{br_2-r_3},\FNorm{br_3-r_4}\leq q^{k-n}\\
&\FNorm{r_2r_4-r_3^2-br_2r_3-br_5},\FNorm{r_1r_4-2r_2r_3+br_2^2-br_1r_3-r_5}\leq q^{k-n} \ .
\end{align*}
\end{Remark}

We are now ready to prove \autoref{Prop:non toral Fourier coefficient}.
\begin{proof}
\begin{itemize}
\item Assume that $\FNorm{b}\less{1}$ i.e. $l\less{0}$. Note that under this assumption
\[
\widehat{U_k\bk{g}}=\overline{U_k\bk{g}}
\]
and thus
\[
E_k\bk{g}=\meas\bk{U_k^{0,0}\bk{g}}-\meas\bk{U_k^{0,1}\bk{g}} -\meas\bk{U_k^{1,0}\bk{g}}+\meas\bk{U_k^{1,1}\bk{g}} \ .
\]
\begin{enumerate}
\item Assume $k=n$, then $u\in\widehat{U_k\bk{g}}$ if and only if
\begin{align*}
&\FNorm{r_2}\leq{q}\\
&\FNorm{r_1},\FNorm{r_5} \leq{q^{2n-m}}\\
&\FNorm{r_3},\FNorm{r_4}, \FNorm{br_1-r_2}, \FNorm{r_2r_4-br_5},\FNorm{r_1r_4-2r_2r_3+br_2^2-br_1r_3-r_5}\leq 1 \ .
\end{align*}
Hence, according to \autoref{lemma:measuring lemma} and arguments performed in the toral case,
\begin{align*}
&\meas\bk{U_k^{0,0}\bk{g}}=q^{-2l} ,
&\meas\bk{U_k^{1,0}\bk{g}}=q^{-2l} \\
&\meas\bk{U_k^{0,1}\bk{g}}=0 ,
&\meas\bk{U_k^{1,1}\bk{g}}=0 \ ,
\end{align*}
and then
\[
E_k\bk{g}=0 \ .
\]

\item Assume $k=n+1$, then $u\in\widehat{U_k\bk{g}}$ if and only if
\begin{align*}
&\FNorm{r_1}\leq q^{1-l}\\
&\FNorm{r_2},\FNorm{r_3},\FNorm{r_4},\FNorm{r_2r_4-r_3^2-br_5},\FNorm{r_1r_4-2r_2r_3-br_1r_3-r_5}\leq q \ .
\end{align*}
Hence, according to \autoref{lemma:measuring lemma} and arguments performed in the toral case,
\begin{align*}
&\meas\bk{U_k^{0,0}\bk{g}}=q^{2-l}\bk{1+\bk{1-q^{-1}}} ,
&\meas\bk{U_k^{1,0}\bk{g}}=q^{2-l}\bk{1+\bk{1-q^{-1}}} \\
&\meas\bk{U_k^{0,1}\bk{g}}=\bk{1-l}q^{2-l}\bk{q-q^{-1}} ,
&\meas\bk{U_k^{1,1}\bk{g}}=\bk{1-l}q^{2-l}\bk{q-q^{-1}} \ ,
\end{align*}
and then
\[
E_k\bk{g}=0 \ .
\]

\item Assume $k\more{n+1}$ and denote $x=r_2+r_3$. Then $u\in\widehat{U_k\bk{g}}$ if and only if
\begin{align*}
&\FNorm{x}\leq{q}\\
&\FNorm{r_5-r_3^2},\FNorm{r_1r_3-r_3^2} \leq{q^{k+n-m}}\\
&\FNorm{r_4},\FNorm{br_1},\FNorm{br_5+r_3r_4-\bk{b+1}r_3^2}, 
\FNorm{r_1\bk{r_4-br_3}+\bk{b+2}r_3^2-2xr_3-r_5}\leq q^{k-n} \ .
\end{align*}
The set $\widehat{U_k\bk{g}}$ is invariant under the change of variables
\[
\bk{r_1,x,r_3,r_4,r_5}\mapsto\bk{r_1,x+\unif^{-1},r_3,r_4,r_5+2r_3\unif^{-1}} \ .
\]
Making this change of variables in the integral yields
\[
E_k\bk{g}=\psi\bk{\unif^{-1}}E_k\bk{g} \ ,
\]
hence
\[
E_k\bk{g}=0 \ .
\]
\end{enumerate}

\item When $\FNorm{b}=\FNorm{d\frac{t_1^2}{t_2}}={1}$ the calculation is more involved and is omitted. Nonetheless $E_k$ vanishes on such elements and hence $D^{\Psi_s}$ also vanishes.
\end{itemize}
\end{proof}

\bibliographystyle{abbrv}
\bibliography{bib}

\begin{thebibliography}{10}

\bibitem{MR525651}
A.~N. Andrianov.
\newblock Multiplicative arithmetic of {S}iegel's modular forms.
\newblock {\em Uspekhi Mat. Nauk}, 34(1(205)):67--135, 1979.

\bibitem{MR2192819}
D.~Bump.
\newblock The {R}ankin-{S}elberg method: an introduction and survey.
\newblock In {\em Automorphic representations, {$L$}-functions and
  applications: progress and prospects}, volume~11 of {\em Ohio State Univ.
  Math. Res. Inst. Publ.}, pages 41--73. de Gruyter, Berlin, 2005.

\bibitem{MR1361787}
D.~Bump, M.~Furusawa, and D.~Ginzburg.
\newblock Non-unique models in the {R}ankin-{S}elberg method.
\newblock {\em J. Reine Angew. Math.}, 468:77--111, 1995.

\bibitem{MR571057}
W.~Casselman.
\newblock The unramified principal series of {$p$}-adic groups. {I}. {T}he
  spherical function.
\newblock {\em Compositio Math.}, 40(3):387--406, 1980.

\bibitem{MR2181091}
W.~T. Gan.
\newblock Multiplicity formula for cubic unipotent {A}rthur packets.
\newblock {\em Duke Math. J.}, 130(2):297--320, 2005.

\bibitem{MR1932327}
W.~T. Gan, B.~Gross, and G.~Savin.
\newblock Fourier coefficients of modular forms on {$G_2$}.
\newblock {\em Duke Math. J.}, 115(1):105--169, 2002.

\bibitem{MR1918673}
W.~T. Gan, N.~Gurevich, and D.~Jiang.
\newblock Cubic unipotent {A}rthur parameters and multiplicities of square
  integrable automorphic forms.
\newblock {\em Invent. Math.}, 149(2):225--265, 2002.

\bibitem{MR1203229}
D.~Ginzburg.
\newblock On the standard {$L$}-function for {$G_2$}.
\newblock {\em Duke Math. J.}, 69(2):315--333, 1993.

\bibitem{UnPublishedGinzburg}
D.~Ginzburg and J.~Hundley.
\newblock A doubling integral for {$G_2$}.
\newblock Preprint.

\bibitem{MR1696481}
B.~H. Gross.
\newblock On the {S}atake isomorphism.
\newblock In {\em Galois representations in arithmetic algebraic geometry
  ({D}urham, 1996)}, volume 254 of {\em London Math. Soc. Lecture Note Ser.},
  pages 223--237. Cambridge Univ. Press, Cambridge, 1998.

\bibitem{MR1637485}
J.-S. Huang, K.~Magaard, and G.~Savin.
\newblock Unipotent representations of {$G_2$} arising from the minimal
  representation of {$D_4^E$}.
\newblock {\em J. Reine Angew. Math.}, 500:65--81, 1998.

\bibitem{MR1617425}
D.~Jiang.
\newblock {$G_2$}-periods and residual representations.
\newblock {\em J. Reine Angew. Math.}, 497:17--46, 1998.

\bibitem{MR965059}
I.~Piatetski-Shapiro and S.~Rallis.
\newblock A new way to get {E}uler products.
\newblock {\em J. Reine Angew. Math.}, 392:110--124, 1988.

\bibitem{MR1020830}
S.~Rallis and G.~Schiffmann.
\newblock Theta correspondence associated to {$G_2$}.
\newblock {\em Amer. J. Math.}, 111(5):801--849, 1989.

\bibitem{MR727854}
N.~R. Wallach.
\newblock Asymptotic expansions of generalized matrix entries of
  representations of real reductive groups.
\newblock In {\em Lie group representations, {I} ({C}ollege {P}ark, {M}d.,
  1982/1983)}, volume 1024 of {\em Lecture Notes in Math.}, pages 287--369.
  Springer, Berlin, 1983.

\end{thebibliography}
\end{document}